\documentclass[]{amsart}
\usepackage[hyperindex=true,bookmarks=true,bookmarksnumbered=true]{hyperref}
\usepackage{enumitem}
\setlist[enumerate]{label=\textnormal{\bfseries(\alph*)}, leftmargin=*, nosep, widest=a}
\setlist[itemize,1]{leftmargin=*}

\usepackage{xcolor}
\usepackage[capitalize]{cleveref} %

\allowdisplaybreaks
\makeatletter
\def\@seccntformat#1{%
  \protect\textup{%
    \protect\@secnumfont
    \expandafter\protect\csname format#1\endcsname %
    \csname the#1\endcsname
    \protect\@secnumpunct
  }%
}
\def\equationautorefname~#1\null{%
  Equation~(#1)\null
}

\newtheoremstyle{convenientthm}%
  {3pt}%
  {3pt}%
  {\itshape}%
  {}%
  {\bfseries}%
  {.}%
  {.5em}%
  {\thmnumber{#2 }\thmname{#1}\thmnote{. #3\addcontentsline{toc}{subsection}{\tocsubsection {}{#2}{#1. #3}}}}%

\newtheoremstyle{convenientplain}%
  {3pt}%
  {3pt}%
  {}%
  {}%
  {\bfseries}%
  {.}%
  {.5em}%
  {\thmnumber{#2 }\thmname{#1}\thmnote{. #3\addcontentsline{toc}{subsection}{\tocsubsection {}{#2}{#1. #3}}}}%

\theoremstyle{convenientthm}

\newtheorem{theorem}[subsection]{Theorem}
\newtheorem*{theorem*}{Theorem}
\newtheorem{lemma}[subsection]{Lemma}

\theoremstyle{convenientplain}

\newtheorem{condition}[subsection]{Conditions on \texorpdfstring{$P$}{P}}
\newtheorem{remark}[subsection]{Remark}

\Crefname{condition}{Condition}{Conditions}

\def\o{\circ}

\def\al{\alpha}
\def\be{\beta}
\def\ga{\gamma}
\def\de{\delta}
\def\ep{\varepsilon}

\def\la{\lambda}

\def\ph{\varphi}

\def\ps{\psi}
\def\om{\omega}
\def\Ga{\Gamma}
\def\De{\Delta}

\def\La{\Lambda}

\def\i{^{-1}}
\def\x{\times}

\let\on=\operatorname

\def\AMSonly#1{}

\def\Id{\on{Id}}

\def\Tr{\on{Tr}}
\def\vol{\on{vol}}
\def\Vol{\on{Vol}}

\def\Diff{\on{Diff}}

\def\Met{\operatorname{Met}}
\def\na{\nabla}

\DeclareMathOperator{\dist}{dist}

\title[Smooth perturbations of the fractional Laplacian]{Smooth perturbations of the functional calculus and applications to Riemannian geometry on spaces of metrics}

\author[M.Bauer, M.Bruveris, P.Harms, PW.Michor]
{Martin Bauer, Martins Bruveris, Philipp Harms, Peter W.~Michor}

\address{Martin Bauer: Faculty for Mathematics, Florida State University, USA}
\email{bauer@math.fsu.edu}

\address{Martins Bruveris: Onfido, 3~Finsbury Avenue, London EC2M 2PA, UK}
\email{martins.bruveris@gmail.com}

\address{Philipp Harms: Faculty for Mathematics, Freiburg University, Germany}
\email{philipp.harms@stochastik.uni-freiburg.de}

\address{
Peter W. Michor: Faculty for Mathematics, University of Vienna, Austria}
\email{peter.michor@univie.ac.at}

\date{\today}

\thanks{We thank Lashi Bandara, Andreas Kriegl, Peer Kunstmann, Gerard Misiolek, Marvin S.~M\"uller, Armin Rainer, and Lutz Weis for helpful discussions. MB was partially supported by NSF-grants 1912037 and 1953244. PH was partially supported in the form of a Junior Fellowship of the Freiburg Institute of Advances Studies.}

\begin{document}

\begin{abstract}
We show for a certain class of operators $A$ and holomorphic functions $f$ that the functional calculus $A\mapsto f(A)$ is holomorphic. 
Using this result we are able to prove that fractional Laplacians $(1+\De^g)^p$ depend real analytically on the metric $g$ in suitable Sobolev topologies.
As an application we obtain local well-posedness of the geodesic equation for fractional Sobolev metrics on the space of all Riemannian metrics. %
\end{abstract}

\maketitle

\section{Introduction}
\label{sec:introduction} 
The space of all Riemannian (Lorentzian, resp.) metrics plays an important role in many areas of pure and applied mathematics and in particular in mathematical physics: it is the natural configuration space for Einstein's equation in general relativity~ \cite{DeWitt67}, it is the central object in Teichm\"uller theory~\cite{fischer1984purely,yamada2014local}, and it appears in the context of mathematical shape analysis~\cite{zhang2021elastic,campbell2021structural}. The above are just a few examples of an extensive list. 

Riemannian geometry on the space of all Riemannian metrics starts with de~Witt \cite{DeWitt67}, who developed a Hamiltonian formulation of general relativity using an $L^2$ metric on the space of all pseudo-Riemannian metrics with fixed signature. 
Ebin \cite{Ebin70} used this metric to prove his slice theorem and, followed by many others, investigated its mathematical properties. 
Clarke \cite{Clarke2013b} showed that the completion of $\Met(M)$ with respect to the $L^2$ geodesic distance contains degenerate and highly irregular metrics. 
This motivated the quest for stronger metrics without this degeneracy, including conformal deformations of the $L^2$ metric \cite{Clarke2013c} and higher order Sobolev metrics \cite{Bauer2013c,smolentsev2007spaces}. 
These developments are reviewed more extensively in \cref{sec:literature:geometry}.

This article establishes local well-posedness of the geodesic equation for a wide class of metrics on the space of all Riemannian metrics. 
A simplified version, for the special case of fractional order Sobolev metrics, of our main result reads as follows:
\begin{theorem*}
On any closed manifold $M$ and for any real number $p\geq1$, the geodesic equation of the weak Riemannian metric 
\begin{equation*}
G_g(h,k)=\int_M \operatorname{Tr}(g\i h g\i(1+\Delta^g)^p   k)\vol(g),\; g\in \Met(M), \;h,k\in T_g\Met(M),    
\end{equation*}
is locally well-posed in the sense of Hadamard. 
\end{theorem*}

Here the Bochner Laplacian $\Delta^g$ acts on symmetric covariant two-tensor fields and is associated to the Levi-Civita connection of the Riemannian metric $g$. 
The theorem follows from the more general \cref{thm:satisfies,thm:wellposed}.
It extends previous well-posedness results for integer order Sobolev metrics (see \cref{sec:literature:mappings}) to more general metrics, including metrics of fractional order.
The proof is an adaptation of the seminal method of Ebin and Marsden \cite{EM1970} for establishing local well-posedness of the incompressible Euler equation.
The adaptation is necessary because the action of diffeomorphisms on metrics differs from the right-action of diffeomorphisms considered in fluid mechanics.
The main difficulty is to show that the geodesic spray is smooth in suitable Sobolev topologies. 
This is a consequence of the following new perturbative result for fractional Laplacians:

\begin{theorem*}
Let $M$ be a closed manifold of dimension $m$,
let $\al\in (m/2,\infty)$ with $\al>1$, 
let $E$ be a natural first order vector bundle over $M$, 
let $\De^g$ be the Bochner Laplacian on $E$ induced by a Riemannian metric $g$,
and let $s,s-2p\in[-\al,\al]$.
Then the following map is real analytic:
\begin{equation*}
\Met_{H^{\al}}(M) \ni g \mapsto (1+\Delta^g)^p \in L(\Ga_{H^s}(E),\Ga_{H^{s-2p}}(E)).
\end{equation*}
\end{theorem*}

Here $\Met_{H^{\al}}(M)$ is the cone of Riemannian metrics on $M$ with Sobolev regularity $H^\al$.
The theorem is a special case of \cref{thm:perturbations} below, which covers fractional powers and more general holomorphic functions of Bochner Laplacians on arbitrary tensor bundles with symmetries. 
The operators are defined using the Levi-Civita connection of the metric and may have coefficients with less regularity than the Dirac and divergence form operators studied in previous related work; see \cref{sec:literature:perturbations}. 
Moreover, these operators are considered not just as unbounded operators on $L^2$, but on an entire scale of fractional Sobolev spaces of positive and negative regularity, with particular attention to the boundary cases of minimal and maximal regularity. 
The above perturbative result is easily seen for integer powers $p$. 
To extend it to non-integer powers $p$, we show that the functional calculus is real analytic in the following sense:

\begin{theorem*}
Let $A$ be a densely defined invertible R-sectorial operator with bounded $\mathcal H^\infty$ calculus on a complex Banach space $X$, 
and let $(\dot X_r)_{r\in \mathbb R}$ be the fractional domain spaces associated to $A$. 
Then the following map is well-defined and holomorphic near $A$ for all $\be<\ga$ and $s,s+r\in[\be,\ga+1]$:
\begin{equation*}
L(\dot X_{\be+1},\dot X_\be)\cap L(\dot X_{\ga+1},\dot X_\ga) \ni B \mapsto B^{-r} \in L(\dot X_s,\dot X_{s+r}).
\end{equation*}
\end{theorem*}

This is a special case of \cref{thm:smooth} below, which is formulated for a more general class of holomorphic functions $f(B)$ instead of fractional powers $B^{-r}$. 
This \lcnamecref{thm:smooth} unifies a series of earlier results for special classes of operators and perturbations, and also strengthens them by linking the domain and range of the operator $f(B)$ to the growth or decay of the function $f$; see \cref{sec:literature:perturbations}.
The proof is based on resolvent integral representations of the functional calculus, as pioneered in the study of perturbations of eigenvalues and eigenvectors by Rellich and Kato \cite{RellichI-V,Kato76}.
The notion of R-sectoriality, which appears in the statement of the theorem, is a generalization of the more widely known notion of sectoriality, and coincides with sectoriality on Hilbert spaces.
Further key tools are perturbative results for operators with bounded $\mathcal H^\infty$ calculus \cite{denk2004new, kalton2006perturbation} and convenient calculus~\cite{KM97}.

\subsection{Context of the results}
\label{sec:literature}

This paper contributes to several different fields: Riemannian geometry on spaces of metrics, well-posedness of geometric partial differential equations, and perturbative operator theory.
We next describe these contributions in the context of previous work and highlight potential future applications and developments. 

\subsubsection{Riemannian geometry on the space of all Riemannian metrics}
\label{sec:literature:geometry}

The study of Riemannian metrics on the space $\Met(M)$ of all Riemannian metrics has a rich history. 
De~Witt \cite{DeWitt67} developed a Hamiltonian formulation of general relativity by writing down for the first time the canonical $L^2$-metric on the space of all pseudo-Riemannian metrics of fixed signature, even splitting it into the trace-free part and the trace part; see also \cite{Pekonen1987, fischer1972einstein} and the references therein. 
Motivated by these physical applications, the mathematical properties of this metric on $\Met(M)$ have been studied in detail:
Ebin \cite{Ebin70} proved the slice theorem. Freed and Groisser \cite{FreedGroisser89} described the geodesics and curvature. The article \cite{Gil-MedranoMichor91} extended this to non-compact manifolds and also described the Jacobi fields and 
the exponential mapping. This was extended to the space of non-degenerate bilinear structures on $M$
in \cite{Gil-MedranoMichorNeuwirther92} and restricted to the space of almost Hermitian structures 
in \cite{Gil-MedranoMichor94}. 
Clarke showed that the geodesic distance for the $L^2$-metric is a positive topological metric on $\Met(M)$ and determined that the metric completion of $\Met(M)$ contains degenerate and highly irregular metrics; see \cite{Clarke2010,Clarke2011,Clarke2013a,Clarke2013b}. 

This motivated the study of stronger Riemannian metrics on $\Met(M)$: Clarke \cite{Clarke2013c} first introduced a conformal deformation of the $L^2$-metric such that the degenerate zero metric does not belong to the metric completion. 
The paper \cite{Bauer2013c} then studied a natural family of stronger $\Diff(M)$-invariant Riemannian metrics on $\Met(M)$ and proved that the geodesic equation is locally well-posed for the integer order Sobolev type metrics considered there. See also the review paper of Smolentsev \cite{smolentsev2007spaces}. 
There was, however, a significant gap in the proof of ~\cite{Bauer2013c}: it was not checked if the geodesic spray extends smoothly to Sobolev completions of $\Met(M)$. 
One of the main result of the present article fills this gap (\cref{thm:perturbations})
and extends the well-posedness result from integer-order Sobolev metrics to a far more general class of metrics, including metrics of fractional order.

In addition the results of this paper  provide an integral step towards constructing a Riemannian metric on $\Met(M)$ such that the completion does not contain any degenerate metrics. This would have direct implications for Riemannian geometry on shape spaces of surfaces, where the question of completeness is still open and of fundamental importance in data-analytic applications; see e.g. the overview article \cite{bauer2014overview}. Furthermore, we hope that our results will be of use in general relativity, e.g., on the space of Riemannian metrics on a Cauchy surface containing initial conditions for solving Einstein's equation. This will require an extension of our results to non-compact manifolds $M$, which is left open for future work.

\subsubsection{Well-posedness of EPDiff equations}
\label{sec:literature:mappings}

In their pioneering work \cite{EM1970} Ebin and Marsden studied the incompressible Euler equation of fluid dynamics by viewing it in Arnold's geometric picture \cite{Arnold1966} as a geodesic equation on a group of diffeomorphisms. 
This led to a proof of local well-posedness, which has subsequently been adapted to a variety of other settings, including the Camassa--Holm \cite{camassa1993integrable,kouranbaeva1999camassa}, Constantin--Lax--Majda \cite{constantin1985simple,escher2012geometry,bauer2016geometric} and EPDiff equations~\cite{holm2005momentum}, as well as several Riemannian structures on spaces of immersions which appear in the context of shape analysis \cite{michor2007overview,Bauer2011b}.

The Ebin--Marsden approach requires an extension of the geodesic spray to a smooth vector field on appropriate Sobolev completions of sufficiently high order. 
This allows one to view the geodesic equation as a flow equation with respect to a smooth vector field, an ODE, and therefore one obtains local existence and uniqueness using the theorem of Picard--Lindel\"of. 
The main difficulty in this approach is to show that all involved operators extend smoothly to the corresponding Sobolev completions.  These operators typically depend on the foot point only via the pull-back metric. 
For this reason, the space of Riemannian metrics is the fundamental object for proving these results. 

The results of this article are the basis for establishing these smoothness properties for a wide class of operators and spaces (including fractional Sobolev metrics on diffeomorphisms, immersions, or densities) and have already led to new well-posedness results; see the follow-up article \cite{bauer2020fractional}. We want to emphasize that these results are relatively easy to prove for differential operators~\cite{EM1970,misiolek2010fredholm}, but highly non-trivial beyond this class, e.g., for pseudo differential operators; see~\cite{escher2014right,bauer2015local,bauer2020well}.
In all of these settings, the generalization from integer to fractional order metrics allows a more fine-grained look at the relation between analytic and geometric properties of Riemannian mapping spaces and their geodesic equations, following e.g.~\cite{ebin2006singularities, misiolek2010fredholm, bauer2018vanishing}. 

\subsubsection{Perturbation theory for linear operators.}
\label{sec:literature:perturbations}

The systematic study of perturbation problems for parameterized families of unbounded self-adjoint or normal operators in a Hilbert space with common domain of definition and compact resolvent has been initiated by Rellich in a series of papers~\cite{RellichI-V}; see also his monograph \cite{Rellich69}.
This theory culminated in Kato's monograph \cite{Kato76}. 
The main tool in their analysis is the resolvent integral, which allows one to obtain perturbative results for eigenvalues and eigenvectors. In particular, Rellich  showed that eigenvalues and eigenvectors can be parameterized real analytically along real analytic curves of self-adjoint operators. However, in general the eigenvalues cannot be chosen smoothly and the eigenfunctions not even continuously as functions of the operator. Nevertheless, by \cite{KMR}, the increasingly ordered eigenvalues are Lipschitz continuous. Further recent contributions can be found in \cite{AKLM98, KM03, KMR, KMRp, RainerAC, RainerN}. 

Perturbations of nonlinear functions of operators have been studied first in the context of Kato's square root problem, which comprises the identification of the domain of the square root of an operator and continuous dependence on parameters \cite{auscher2002solution,mcintosh1990square,auscher1998square,Kato76}.
While resolvent integrals are still a key tool, a major difficulty is that their convergence requires some extra regularity, which forces one to work in a weaker topology than one would optimally desire. 
This can be seen in \cref{lem:holomorphic}.\ref{lem:holomorphic:a}--\ref{lem:holomorphic:c} below and in several perturbative results for fractional powers of selfadjoint operators associated to sesquilinear forms; see e.g.~\cite[Theorem~2.5]{yagi1983differentiability} or \cite[Theorem 6.1]{axelsson2006quadratic}. 
To avoid this loss of regularity it seems necessary to impose some additional bounds on imaginary powers of operators or more general bounded holomorphic functions of operators \cite{mcintosh1985square, yagi1987applications, axelsson2006quadratic}.
For example, Dirac operators and divergence form operators can be shown to have bounded $\mathcal H^\infty$ calculus uniformly on an $L^\infty$ neighborhood of their coefficients \cite{axelsson2006quadratic, morris2012kato,  bandara2016kato, bandara2017continuity, bandara2016rough}.
In the case of general sectorial operators, neither bounded imaginary powers nor boundedness of the $\mathcal H^\infty$ functional calculus are stable under relatively bounded perturbations. 
However, boundedness of the $\mathcal H^\infty$ functional calculus is stable under perturbations in two distinct fractional domain scales \cite{kalton2006perturbation}. 
This is a key ingredient to our general result on holomorphic perturbations of the functional calculus (see \cref{thm:smooth}).

\subsection{Structure of the article}
\cref{sec:preliminaries} sets up some notation and lists some external results, which are used extensively throughout the article.
\cref{sec:laplacians} establishes the real analytic dependence of the Bochner Laplacian on the Riemannian metric. 
\cref{sec:sectorial} contains our general result on holomorphic perturbations of the functional calculus. 
\cref{sec:perturbative} applies this general result to Laplacians on closed Riemannian manifolds.
\cref{sec:metrics} shows the local well-posedness of the geodesic equation for fractional order Sobolev metrics on the manifold of all Riemannian metrics.

\section{Preliminaries}
\label{sec:preliminaries}
\subsection{Setting}
\label{sec:setting}

We use the notation of \cite{Bauer2013c} and write $\mathbb N$ for the natural numbers including zero. 
Smooth will mean $C^\infty$ and real analytic $C^\om$.
Real vector spaces and their complexifications will not be distinguished notationally.
Sobolev spaces induced by Riemannian metrics $g$ are denoted by $H^s(g)$, $s\in \mathbb R$ (see \cref{sec:domain}). 
If $g$ has finite Sobolev regularity, they coincide with the standard Sobolev spaces $H^s$ for a restricted range of $s$ (see \cref{lem:fractional}). 

Throughout this paper, without any further mention, we fix a  smooth connected closed (i.e., compact without boundary) manifold $M$ of dimension $m \in \mathbb N_{>0}$. 

\subsection{First order natural bundles}
\label{sec:natural}
A first order natural bundle over $M$ is a smooth vector bundle $E\to M$ associated to the first order frame bundle of $M$ with respect to some representation of $GL(m)$.
The completely reducible representations give rise to exactly the tensor bundles and their subbundles  which are described by symmetries: the irreducible ones corresponding to Young tableaus, possibly tensored by a bundle of  $p$-densities
$|\La^m|^p(T^*M)$ for $p\in \mathbb R$.
Examples are trivial bundles, $TM$, $T^*M$, $S^2T^*M$, $\La^kT^*M$, and the bundles following the algebraic symmetries of Riemannian curvatures. 
See \cite{KMS93} for a treatment of natural bundles and \cite{Fulton1997} for a description of Young tableaus.
\emph{In this paper, by a first order natural bundle we shall mean one which is induced by a completely reducible representation; i.e., the tensor bundles and their subbundles described above.}

\subsection{Sobolev spaces}
\label{sec:sobolev}

We write $H^s(\mathbb R^m,\mathbb R^n)$ for the Sobolev space of order $s\in\mathbb R$ of $\mathbb R^n$-valued functions on $\mathbb R^m$. 
We will now generalize these spaces to sections of vector bundles.  
Let $E$ be a vector bundle of rank $n\in\mathbb N_{>0}$ over $M$ and let $\Ga(E)$ denote the corresponding space of (smooth) sections. 
We choose a finite vector bundle atlas and a subordinate partition of unity in the following way. 
Let $(u_i:U_i \to u_i(U_i)\subseteq \mathbb R^m)_{i\in I}$ be a finite atlas for $M$, 
let $(\ph_i)_{i\in I}$ be a smooth partition of unity subordinated to $(U_i)_{i \in I}$, and let $\ps_i:E|U_i \to U_i\x \mathbb R^n$ be vector bundle charts. 
Note that we can choose open sets $U_i^\o$ such that $\on{supp}(\ps_i)\subset U_i^\o\subset \overline{U_i^\o}\subset U_i$  and each $u_i(U_i^\o)$ is an open set in $\mathbb R^m$ with Lipschitz boundary
(cf.\@ \cite[Appendix~H3]{behzadan2017certain}).
Then we define for each $s \in \mathbb R$ and $f \in \Ga(E)$
\begin{equation*}
\|f\|_{\Ga_{H^s}(E)}^2 := \sum_{i \in I} \|\on{pr}_{\mathbb R^n}\o\, \ps_i\o (\ph_i \cdot f)\o u_i\i \|_{H^s(\mathbb R^m,\mathbb R^n)}^2,
\end{equation*}
where $\on{pr}_{\mathbb R^n}$ denotes the projection onto $\mathbb R^n\subset \mathbb R^m$.
Then $\|\cdot\|_{\Ga_{H^s}(E)}$ is a norm, which comes from a scalar product, and we write $\Ga_{H^s}(E)$ for the Hilbert completion of $\Ga(E)$ under the norm. 
It turns out that $\Ga_{H^s}(E)$ is independent of the choice of atlas and partition of unity, up to equivalence of norms. We refer to \cite[Section~7]{triebel1992theory2} and \cite[Section~6.2]{grosse2013sobolev} for further details. 

In this article we only consider Sobolev spaces $H^{\al}$. Most of the results carry over with suitable modifications to other scales of complex interpolation spaces, including scales of Bessel potential spaces $H^{\al,p}$. 
Another possible generalization is to replace the compact manifold $M$ by an open manifold and use Sobolev spaces measured by a smooth background Riemannian metric of bounded geometry on $M$, in the spirit of Eichhorn \cite{Eichhorn2007}.

\begin{theorem}[Module properties of Sobolev spaces]
\label{thm:module}
Let $E_1,E_2$ be vector bundles over $M$
and let $s_1,s_2,s\in\mathbb R$ satisfy  
\begin{enumerate}
\item
\label{thm:module:a}
$s_1+s_2\geq 0$, $\min(s_1,s_2)\geq s$, and $s_1+s_2-s>\frac m 2$, or
\item 
\label{thm:module:b}
$s \in \mathbb N$, $\min(s_1,s_2)>s$, and $s_1+s_2-s\geq\frac m 2$, or
\item 
\label{thm:module:c}
$-s_1 \in \mathbb N$ or $-s_2 \in \mathbb N$, $s_1+s_2>0$, $\min(s_1,s_2)>s$, $s_1+s_2-s \geq \frac{m}{2}$.
\end{enumerate}
Then the tensor product of smooth sections extends to a bounded bilinear mapping 
\begin{equation*}
\Ga_{H^{s_1}}(E_1)\x \Ga_{H^{s_2}}(E_2) \to \Ga_{H^{s}}(E_1\otimes E_2).
\end{equation*}
\end{theorem}

\begin{proof}
Recall that $H^s=W^{s,2}=H^{s,2}$. 
Thanks to the local description of Sobolev spaces in  \cref{sec:sobolev} it suffices to consider compactly supported functions and distributions on $\mathbb R^m$. 
The sufficiency of condition~\ref{thm:module:a} in the case $s\geq 0$ follows from \cite[Th\'eor\`eme~2]{zolesio1977multiplication} or \cite[Theorem~5.1]{behzadan2015multiplication} or \cite[Theorem~7.3]{behzadan2015multiplication}.
Duality allows one to replace $(s_1,s_2,s)$ by $(s_1,-s,-s_2)$ or $(-s,s_2,-s_1)$, which implies the sufficiency of condition~\ref{thm:module:a} in the case $\min(s_1,s_2)\leq 0$. 
The sufficiency of condition~\ref{thm:module:a} in the remaining case $s<0<\min(s_1,s_2)$ has been shown in \cite[Theorem~8.3]{behzadan2015multiplication}.
The sufficiency of condition~\ref{thm:module:b} follows from \cite[Th\'eor\`eme~2]{zolesio1977multiplication} or \cite[Theorem~6.1]{behzadan2015multiplication}.
The sufficiency of condition~\ref{thm:module:c} follows by duality from condition~\ref{thm:module:b}.
\end{proof} 

Note that the conditions of the above \lcnamecref{thm:module} are invariant under multiplication and duality. Indeed, letting $p(s_1,s)$ denote the set of all $s_2$ such that $(s_1,s_2,s)$ satisfies condition \ref{thm:module:a}, \ref{thm:module:b}, or \ref{thm:module:c} of \cref{thm:module}, one easily verifies that the following statements hold for all $r,s,t \in \mathbb R$: 
\begin{itemize}
\item If $\al \in p(r,s) $ and $\beta \in p(s,t)$, then $\min(\al,\beta) \in p(r,t)$, and the tensor product of smooth sections extends to a bounded bilinear mapping 
\begin{equation*}
\Ga_{H^{\al}}(E_1)\x \Ga_{H^{\beta}}(E_2) \to \Ga_{H^{\min(\al,\beta)}}(E_1\otimes E_2).
\end{equation*}

\item If $\beta \in p(r,s)$, then $\beta \in p(-s,-r)$. 
\end{itemize}	

\subsection{Convenient calculus}
\label{sec:convenient}

We will make essential use of convenient calculus as developed in \cite{FK88} and \cite{KM97}. 
A locally convex vector space $X$ is called {\sl convenient} if each Mackey Cauchy sequence has a limit; equivalently, if for each smooth curve $c\colon\mathbb R\to X$ the Riemann integral $\int_0^1 c(t)\, dt$ converges.
This property and those mentioned below depend only on the system of bounded sets in $X$. 
Every Banach and Fr\'echet space is convenient. 
Moreover, by \cite[Theorem~2.15]{KM97} the following constructions preserve convenient vector spaces: limits, direct sums, and strict inductive limits of closed embeddings; 
this is needed in \cref{lem:holomorphic}.
Mappings between convenient vector spaces are called smooth if they map smooth curves to smooth curves. 
A smooth mapping is real analytic if it is real analytic along each affine line. 
A mapping is holomorphic if it is holomorphic along each holomorphic map from the unit disk in $\mathbb C$ to $X$, or even along each affine complex line. 
 
We will make essential use of the following properties \cite{KM97}.
\begin{enumerate}
\item 
\label{sec:convenient:a}
Smooth (or real analytic or holomorphic) curves can be recognized if they remain so after applying each bounded linear functional in a subset of the dual which is large enough to recognize bounded subsets.
\item 
\label{sec:convenient:b}
Convenient smoothness coincides with all other notions of $C^\infty$ up to Fr\'echet spaces. Moreover, convenient real analyticity and holomorphicity coincides with all other notions of $C^\om$ and $\mathcal H^\infty$ up to Banach spaces. 
\item
\label{sec:convenient:c}
 If $X$ and $Y$ are convenient, then the space $L(X,Y)$ of bounded linear operators between $X$ and $Y$ is convenient. Moreover, the following uniform boundedness theorem hold true: an $L(X,Y)$-valued map is smooth if and only if all its evaluations against $x \in X$ are smooth. Similar statements hold with smooth replaced by real analytic; this is called real analytic uniform boundedness theorem~ \cite[Theorem~11.14]{KM97}. This follows from \ref{sec:convenient:a} and the classical linear uniform boundedness theorem.
\end{enumerate}

\section{Laplacians associated to Riemannian metrics of finite Sobolev regularity} 
\label{sec:laplacians}

This section develops the theory of Riemannian metrics of finite Sobolev regularity and their induced Laplacians on general first order natural vector bundles, which are exactly tensor bundles with symmetries. 
The main results are on functional analytic properties of the Laplacians and on real analytic perturbations of the metric. 

\subsection{Metrics of Sobolev order}
\label{sec:sobolev_metrics}

The bundle of symmetric covariant two-tensors is denoted by $S^2T^*M$ and the subbundle of positive definite tensors by $S^2_+T^*M$. 
Then the space of smooth Riemannian metrics is the space $\Ga(S^2_+T^*M)$ of smooth sections.
Moreover, for any $\alpha \in (m/2,\infty)$, the space of Riemannian metrics of Sobolev regularity $\alpha$ is the space of $H^\al$-sections:
\begin{equation*}
\operatorname{Met}_{H^\al}(M):= \Ga_{H^\al}(S^2_+T^*M).
\end{equation*}
This is well-defined because the condition $\al>m/2$ ensures that the tensors in $\Ga_{H^\al}(S^2T^*M)$ are continuous and that $\operatorname{Met}_{H^\al}(M)$ is an open subset of the space $\Ga_{H^\al}(S^2T^*M)$. 
More generally, a fiber metric of regularity $\al$ on a vector bundle $E$ is an element of $\Ga_{H^\al}(S^2_+E^*)$.

\begin{lemma}[Inverse metric]
\label{lem:inverse}
For any $\al\in(m/2,\infty)$, the inverse metric is well-defined and real analytic as a mapping
\begin{equation*}
\Met_{H^\al}(M)\ni g \mapsto g\i \in \Ga_{H^\al}(S^2_+TM).
\end{equation*}
\end{lemma}

\begin{proof}
The inverse metric satisfies the implicit equation $\on{Tr}_{2,3}(g\otimes g\i)=\on{Id}_{TM}$, where the trace $\on{Tr}_{2,3}$ contracts the second and third tensor slot. The left-hand side of this implicit equation is real analytic in $g$ and $g\i$ because the tensor product 
\begin{align*}
\otimes\colon \Ga_{H^\al}(T^*M\otimes T^*M) \x \Ga_{H^\al}(TM\otimes TM) &\to \Ga_{H^\al}(T^*M\otimes T^*M \otimes TM\otimes TM)
\end{align*}
is bounded bilinear by the module property of \cref{thm:module}, and the trace 
\begin{align*}
\on{Tr}_{2,3}\colon \Ga_{H^\al}(T^*M\otimes T^*M \otimes TM\otimes TM) \to \Ga_{H^\al}(T^*M\otimes TM)
\end{align*}
is bounded linear. Thus, it follows from the real analytic implicit function theorem for Banach spaces that $g\mapsto g\i$ is real analytic.
\end{proof}

\begin{lemma}[Volume form and duality]
\label{lem:volume}
Let $\al\in(m/2,\infty)$, 
and let $\Vol M=|\La^m|(T^*M)$ denote the volume bundle.
Then the following statements holds:
\begin{enumerate}
\item
\label{lem:volume:a}
The Riemannian volume form is well-defined and real analytic as a mapping
\begin{equation*}
\Met_{H^\al}(M) \ni g \mapsto \vol^g \in \Ga_{H^\al}(\Vol M).
\end{equation*}

\item 
\label{lem:volume:b}
For any $g \in \Met_{H^\al}(M)$, the pairing $(h,k)\mapsto \int g(h,k)\vol^g$ extends for all $s \in [-\al,\al]$ to a bounded bilinear map 
\begin{equation*}
\langle\cdot,\cdot\rangle_{H^0(g)}\colon \Ga_{H^s}(E)\times\Ga_{H^{-s}}(E)\to\mathbb R,
\end{equation*}
which induces a topological isomorphism $\Ga_{H^{-s}}(E) \to (\Ga_{H^s}(E))^*$, called the $H^0(g)$-duality. 

\item This duality is real analytic as a mapping
\begin{equation*}
\Met_{H^\al}(M) \ni g \mapsto \langle\cdot,\cdot\rangle_{H^0(g)} \in L(\Ga_{H^{-s}}(E), (\Ga_{H^s}(E))^*).
\end{equation*}
\end{enumerate}
\end{lemma}

\begin{proof}
\begin{enumerate}[wide]
\item As explained in \cref{sec:sobolev}, the space of $H^\al$-sections is described locally. Thus, we may replace $M$ by an open subset $U$ of $\mathbb R^m$ and use the following expression of the volume form in local coordinates $(x^1,\dots,x^m)$:
\begin{equation*}
\vol^g = \sqrt{\operatorname{det} \big((g_{i,j})_{i,j=1}^m\big)} dx^1\dots dx^m. 
\end{equation*}
The determinant $\Met_{H^\al}(U)\ni g\mapsto \det(g_{i,j})_{i,j=1}^m \in H^\al(U)$ is real analytic by the module property of \cref{thm:module}. 
The square root $H^\al(U,\mathbb R_{>0})\ni f\mapsto \sqrt{f}\in  H^\al(U,\mathbb R_{>0})$ is real analytic, again by the real analytic implicit function theorem on Banach spaces. 
Therefore, $g\mapsto\vol^g$ is real analytic, as claimed.

\item
The statement holds for any smooth Riemannian metric $\hat g \in \Met(M)$. 
The pairings induced by $g$ and $\hat g$ are related as follows: for any $h,k\in\Ga_{H^s}(E)$,
\begin{equation*}
\int_M g(h,k)\vol^g = \int_M \hat g\left(\frac{\vol^g}{\vol^{\hat g}}\ \hat g^{-1}gh,k\right)\vol^{\hat g}.
\end{equation*}
Here $\vol^g/\vol^{\hat g} \in H^\al(M,\mathbb R)$ denotes the Radon-Nikodym derivative, whose coordinate expression can be seen from \ref{lem:volume:a}.
The linear operator 
\begin{equation*}
\Ga_{H^s}(E) \ni h\mapsto \frac{\vol^g}{\vol^{\hat g}}\ \hat g^{-1}gh \in \Ga_{H^s}(E).
\end{equation*}
is bounded with bounded inverse by \cref{thm:module}. 
This proves \ref{lem:volume:b}.

\item The operator in the last displayed equation depends real analytically on $g \in \Met_{H^\al}(M)$.
\qedhere
\end{enumerate}
\end{proof}

The following \lcnamecref{lem:fiber} generalizes the constructions of \cref{lem:inverse,lem:volume} to arbitrary first order natural bundles.

\begin{lemma}[Induced fiber metrics]
\label{lem:fiber}
Let $g \in \Met_{H^\al}(M)$ be a Riemannian metric of Sobolev regularity $\al\in(m/2,\infty)$, and let $E$ be a first order natural bundle over $M$. 
\begin{enumerate}
\item 
\label{lem:fiber:a}
The metric $g$ induces a canonical fiber metric of class $H^\al$ on $E$ (up to the choice of some constants). 

\item 
\label{lem:fiber:b}
The fiber metric can be chosen real analytically in $g$, yielding a real analytic map $\Met_{H^\al}(M)\to \Ga_{H^\al}(S^2_+E^*)$. 

\item 
\label{lem:fiber:c}
If $E$ is trivial, then the fiber metric is of class $C^\infty$ and does not depend on $g$.
\end{enumerate}
\end{lemma}

\begin{proof}
\begin{enumerate}[wide]
\item If $E=T^r_sM$ is a tensor bundle of contravariant rank $r\in\mathbb N$ and covariant rank $s\in \mathbb N$, then $E$ inherits the canonical metric $g^{\otimes r}\otimes (g^{-1})^{\otimes s}$ from $T^r_sM$.
More generally, if $E$ is a subbundle of $T^r_sM$ described by some symmetries (cf.~\cref{sec:natural}), then the canonical metric is the restriction of $g^{\otimes r}\otimes (g^{-1})^{\otimes s}$ to $E$. 
On the line bundle $|\La^m|^p(T^*M)$ of $p$-densities, $(\vol^g)^{-2p}$ is the induced metric.
In general, $E$ can be identified with a direct sum of tensor bundles with symmetries, each tensored with a line bundle of $p$-densities,
and the canonical metric is defined accordingly as a direct sum of metrics as above. 
Different identifications lead to metrics which coincide up to a constant on each irreducible component of the representation describing $E$ as an associated bundle. 
In any case, the canonical fiber metric is of class $H^\al$ if $g$ is of class $H^\al$ thanks to the module property of Sobolev spaces, \cref{lem:volume}, and the assumption that $\al$ is above the Sobolev threshold $m/2$. 

\item 
We identify $E$ with a tensor bundle with symmetries and choose the fiber metric as in the first part of \ref{lem:fiber:a}, thereby eliminating the non-uniqueness. 
The multilinear algebra described there reduces the statement to $g\i$ and $(\vol^g)^{-2p}$. 
But real analyticity of the inverse metric and volume form has been shown in \cref{lem:inverse,lem:volume}.
Moreover, the map
\begin{equation*}
\Ga_{H^\al}(|\La^m|(T^*M))\ni\vol^g\mapsto(\vol^g)^{-2p}\in\Ga_{H^\al}(|\La^m|^{-2p}(T^*M))
\end{equation*}
is real analytic because its power series converges on open sets with respect to the supremum norm, which are $H^\al$-open thanks to the Sobolev embedding theorem. 

\item 
The fiber metric on $T^0_0M=M\times\mathbb R$ does not depend on $g$ and is of class $C^\infty$, and any trivial bundle is a direct sum of such bundles.
\qedhere
\end{enumerate}
\end{proof}

By abuse of notation we will sometimes write $g$ for the metric as well as the induced fiber metric of \cref{lem:fiber}. 

\begin{lemma}[Covariant derivative]
\label{lem:covariant}
Let $\al\in(m/2,\infty)$ and $s\in[1-\al,\al]$.
\begin{enumerate}
\item 
\label{lem:covariant:a}
For each $g\in\Met_{H^\al}(M)$ and natural first order vector bundle $E$ over $M$, 
there is a unique bounded linear mapping
\begin{equation*}
\Ga_{H^s}(E) \ni h \mapsto \nabla^g h \in \Ga_{H^{s-1}}(T^*M\otimes E)
\end{equation*}
which acts as a derivation with respect to tensor products, commutes with each symmetrization operator, and coincides with the Levi-Civita covariant derivative in the cases $E=TM$ and $E=T^*M$. 
\item
\label{lem:covariant:b}
The covariant derivative is real analytic as a mapping 
\begin{align*}
\Met_{H^\al}(M) \ni g \mapsto \nabla^g \in  L(\Ga_{H^s}(E),\Ga_{H^{s-1}}(T^*M\otimes E)).
\end{align*}

\item 
\label{lem:covariant:c}
If $E$ is trivial, then this holds for all $s \in \mathbb R$.
\end{enumerate}
\end{lemma}

We will show this \lcnamecref{lem:covariant} in two ways.

\begin{proof}
\begin{enumerate}[wide]
\item 
Assume temporarily that $E=TM$. 
Let $X$, $Y$, and $Z$ be arbitrary smooth vector fields. 
Let $\widehat\nabla$ be a smooth covariant derivative on $M$, either induced via charts or by a fixed smooth background Riemannian metric. 
We express the Levi-Civita connection of $g\in \Met_{H^{\al}}(M)$ as 
\begin{equation*}
\nabla^g_XY = \widehat\nabla_XY + A^g(X,Y)
\end{equation*}
for a suitable section $A^g$ of the bundle $T^1_2M = T^*M\otimes L(TM,TM)$. 
As $\nabla^g$ has to be $g$-compatible and torsion-free, the tensor field $A^g$ has to satisfy the following conditions:
\begin{align*}
(\widehat\nabla_X g)(Y,Z) &= g(A^g(X,Y),Z) + g(Y,A^g(X,Z)),
\\
A^g(X,Y) &= A^g(Y,X). 
\end{align*}
We take the cyclic permutations of the first equation, sum them with signs $+,+,-$, and use symmetry of $A^g$ to obtain
\begin{equation*}
2g(A^g(X,Y),Z) = (\widehat\nabla_X g)(Y,Z) + (\widehat\nabla_Y g)(Z,X) - (\widehat\nabla_Z g)(X,Y);
\end{equation*}
this equation determines $A^g$ uniquely as a $H^{\al-1}$-tensor field. 
It is easily checked that it satisfies the two requirements above (i.e., compatibility with the metric and freedom of torsion). 
Together with the module properties of \cref{thm:module} this implies \ref{lem:covariant:a} for $E=TM$.

The extension from $E=TM$ to first order natural bundles $E$ is via multilinear algebra: the case $E=T^*M$ follows by testing with smooth vector fields, and the general case by choosing local frames for $E$, writing any $H^s$ section of $E$ as a sum of scalar $H^s$ coefficients times $C^\infty$ sections of $E$, and using the module property of \cref{thm:module}.
For the line bundle of $p$-densities the covariant derivative is simply determined by $\nabla^g_X(\vol^g)^p=0$ for all $p\in \mathbb R$, i.e., 
\begin{equation*}
\nabla^g_X\nu = d\Big(\frac{\nu}{(\vol^g)^p}\Big)(X)\cdot(\vol^g)^p\,.
\end{equation*}

\item 
As before it is sufficient to treat the case $E=TM$. 
The tensor field $A^g$ is given in abstract index notation by
\begin{align*}
(A^g)^i_{jk}=\frac12g^{il}\big((\widehat\nabla g)_{ljk}+(\widehat\nabla g)_{jkl}-(\widehat\nabla g)_{ljk}\big),
\end{align*}
where $g^{il}$ denotes the inverse of the metric. 
Both $g^{-1}\in \Ga_{H^\al}(S^2TM)$ and $\widehat\nabla g\in\Ga_{H^{s-1}}(T^0_3M)$ are real analytic in $g \in \Met_{H^\al}(M)$, and the contractions between $g^{-1}$ and $\widehat\nabla g$ are continuous by the module properties of \cref{thm:module}. 
Therefore, $g\mapsto A^g$ is real analytic $\Met_{H^\al}(M)\to \Ga_{H^{\al-1}}(T^1_2M)$. 
Together with the module properties of \cref{thm:module} this shows that $\nabla^gY=\widehat\nabla Y+A^g(\cdot,Y) \in \Ga_{H^{\al-1}}(T^1_1M)$ is real analytic in $g$ for each $Y \in \Ga_{H^s}(TM)$.
By the real analytic uniform boundedness theorem \cite[Theorem~11.14]{KM97} this implies that $\nabla^g \in L(\Ga_{H^s}(TM),\Ga_{H^{s-1}}(T^1_1M))$ is real analytic in $g$. 

\item The fiber metric on trivial bundles is smooth and does not depend on $g$ by \cref{lem:fiber}.\ref{lem:fiber:c}.
\qedhere
\end{enumerate}
\end{proof}

Note that \cref{lem:covariant} and the module property of \cref{thm:module} imply that the covariant derivative is a bounded bilinear mapping
\begin{equation*}
\nabla^g:\Ga_{H^{\al}}(TM)\x \Ga_{H^s}(TM)\ni (X,Y) \mapsto \nabla^g_XY \in \Ga_{H^{s-1}}(TM).
\end{equation*}
For $E=TM$, this can also be seen directly from the defining properties of the Levi-Civita covariant derivative. 

\begin{remark}[Geodesics]
\label{rem:geodesics}
The proof of \cref{lem:covariant} shows that the Christoffel symbols are of  class $H^{\al-1}$. 
They transform as the last part in the second tangent bundle, and the associated spray $S^g$ is an $H^{\al-1}$-section of both $\pi_{TM}:T^2M\to TM$ and $T(\pi_M):T^2M\to TM$; see \cite[Section~22.6]{Michor08}.        
If $\al>\frac{m}2+1$, then the spray $S^g$ is continuous and we have local existence (but not uniqueness) of geodesics in each chart separately, by Peano's theorem. 
If $\al>\frac{m}2+2$, then $S^g$ is $C^1$ and there is existence and uniqueness of geodesics by Picard-Lindel\"of. 
\end{remark}

\begin{theorem}[Bochner Laplacian]
\label{thm:bochner}
Let $\al\in(m/2,\infty)$,
let $s \in [2-\al,\al]$,
and let $E$ be a natural first order vector bundle over $M$.
\begin{enumerate}
\item 
\label{thm:bochner:a}
For each $g \in \Met_{H^\al}(M)$, the Bochner Laplacian is a bounded Fredholm operator of index zero
\begin{equation*}
\De^g: \Ga_{H^s}(E) \ni h \mapsto 
- \on{Tr}^{g\i}(\na^{g}\na^{g} h) \in \Ga_{H^{s-2}}(E). 
\end{equation*}

\item 
\label{thm:bochner:b}
The Laplacian depends real analytically on the metric, i.e., the following mapping is real analytic:
\begin{equation*}
\Met_{H^{\al}}(M) \ni g \mapsto \De^g \in  L(\Ga_{H^{s}}(E),\Ga_{H^{s-2}}(E)).
\end{equation*}

\item 
\label{thm:bochner:c}
If $E$ is trivial then these statements hold for all $s \in [2-\al,\al+1]$.
This also applies to Laplace operators induced by connections which do not depend on the metric $g$. 
\end{enumerate}
\end{theorem}

Similar statements for more general differential operators with Sobolev coefficients can be found in \cite[Theorem~2]{Mueller2017} and \cite[Lemma~34]{holst2009rough}; see also \cite{bandara2016kato} and references therein for operators in divergence form. 

\begin{proof}
\begin{enumerate}[wide]
\item 	
By \cref{thm:module} and \cref{lem:covariant} the Laplacian is a bounded linear mapping
\begin{equation*}
\Delta^g\colon \Ga_{H^s}(E) \stackrel{\nabla^g}{\longrightarrow} \Ga_{H^{s-1}}(T^*M\otimes E) \stackrel{\nabla^{g}}{\longrightarrow} \Ga_{H^{s-2}}(T^0_2M\otimes E) \stackrel{\Tr^{g^{-1}}}{\longrightarrow} \Ga_{H^{s-2}}(E),
\end{equation*}
where the first arrow requires $s \in [1-\al,\al]$, the second arrow requires $s-1 \in [1-\al,\al]$, and the third arrow requires $s-2\in[-\al,\al]$. 
Integration by part shows that $\Delta^g$ is formally self-adjoint with respect to the $H^0(g)$ inner product, and a similar reasoning as above shows that the $H^0(g)$-adjoint operator (see \cref{lem:volume})
\begin{equation*}
(\Delta^g)^{*,H^0(g)}=\Delta^g\colon \Ga_{H^{2-s}}(E) \to \Ga_{H^{-s}}(E)
\end{equation*}
is also a bounded linear mapping. 

The Sobolev regularity of the coefficients of $\Delta^g$ is linked to the number of derivatives as follows: in each vector bundle chart of $E$, the operator $\Delta^g$ takes the form
\begin{equation*}
\Delta^g = \sum_{i=1}^m a^i \partial_{x^i} + \sum_{i,j=1}^m a^{i,j}\partial_{x^i}\partial_{x^j}
\end{equation*}
for some coefficients $a^i \in H^{\al-1}(\mathbb R^m,\mathbb R^{n\times n})$ and $a^{i,j}\in H^{\al}(\mathbb R^m,\mathbb R^{n\times n})$, as can be seen from the proof of \cref{lem:covariant}.
Moreover, the leading-order coefficients $(a^{i,j})$ are coordinate expressions of the cometric and therefore invertible.
Therefore, the elliptic estimates of \cite[Lemmas~32--34]{holst2009rough} may be applied, and it follows for each $s \in (2-\al,\al]$ that the operator $\Delta^g\colon \Ga_{H^s}(E) \to \Ga_{H^{s-2}}(E)$ has finite-dimensional kernel and closed range.
In particular, it is semi-Fredholm, and its index $\operatorname{ind}(\Delta^g) \in [-\infty, \infty)$ is well-defined. 

The set of all semi-Fredholm operators with the same index as $\Delta^g$ is open in $L(\Ga_{H^s}(E),\Ga_{H^{s-2}}(E))$ (see e.g. \cite[Theorem~III.18.4]{muller2007spectral}). By  continuity of the mapping
\begin{equation*}
\Met_{H^\al}(M) \to L(\Ga_{H^s}(E),\Ga_{H^{s-2}}(E)), \qquad g \mapsto \Delta^g,
\end{equation*}
there is a smooth metric $\hat g$ which is sufficiently close to $g$ such that $\Delta^{\hat g}$ is semi-Fredholm and has the same index as $\Delta^g$. 
But standard elliptic theory implies that $\Delta^{\hat g}$ is Fredholm of index zero. 
Thus, we have shown that $\Delta^g\colon \Ga_{H^s}(E) \to \Ga_{H^{s-2}}(E)$ is Fredholm of index zero for each $s \in (2-\al,\al]$. This extends to all $s \in [2-\al,\al]$ by $H^0(g)$-duality. 

\item 
This follows from the real analyticity of the covariant derivative (\cref{lem:covariant}) and the module properties of Sobolev spaces (\cref{thm:module}).

\item Writing the covariant derivative of functions as a differential, one obtains from \cref{lem:covariant} and \cref{thm:module} that the Laplacian is a continuous linear operator
\begin{equation*}
\Delta^g\colon {H^s}(M,\mathbb R) \stackrel{d}{\longrightarrow} H^{s-1}(M,T^*M\otimes \mathbb R) \stackrel{\nabla^{g}}{\longrightarrow} H^{s-2}(M,T^0_2M\otimes \mathbb R) \stackrel{\Tr^{g^{-1}}}{\longrightarrow} H^{s-2}(M,\mathbb R),
\end{equation*}
where the first mapping is continuous for all $s\in \mathbb R$, the second one for  $s-1 \in [1-\al,\al]$, and the third one for  $s-2\in[-\al,\al]$. 
A similar statement holds when the first map $d$ is replaced by a connection which does not depend on $g$, and the second map $\nabla^g$ is the induced connection on $T^*M\otimes E$.

In local coordinates $(x^1,\dots,x^m)$, the Laplacian can be expressed as
\begin{align*}
\De^g f
&=
\sum_{i,j=1}^m \frac{1}{\sqrt  {|\operatorname{det}(g)|}}\partial_{x^i}\left({\sqrt  {|\operatorname{det}(g)|}}g^{{ij}}\partial_{x^j}f\right)
\\&=
\sum_{\substack{i_1,\dots,i_m \in \mathbb N_{\geq 0}\\1\leq i_1+\dots+i_m\leq 2}}  a^{i_1,\dots,i_m} \partial_{x_1}^{i_1}\dots\partial_{x_m}^{i_m},
\end{align*}
where $a^{i_1,\dots,i_m} \in H^{\al-2+i_1+\dots+i_m}(\mathbb R^m,\mathbb R)$. 
Note that there is no zero-order term. 

Assume temporarily that $s \in (2-\al,\al+1]$. 
By \cref{thm:module}, there is $\delta>0$ such that the first-order part 
\begin{equation*}
\sum_{\substack{i_1,\dots,i_m \in \mathbb N_{\geq 0}\\i_1+\dots+i_m=1}}  a^{i_1,\dots,i_m} \partial_{x_1}^{i_1}\dots\partial_{x_m}^{i_m} \colon H^{s-\delta}(M,\mathbb R)\to H^{s-2}(M,\mathbb R)
\end{equation*}
is continuous.  
Therefore, the elliptic estimate of \cite[Lemma~32]{holst2009rough} still holds in the present setting, and it follows from \cite[Lemmas~33--34]{holst2009rough} that $\Delta^g\colon H^s(M,\mathbb R)\to H^{s-2}(M,\mathbb R)$ has finite-dimensional kernel and closed range for each $s \in (2-\al,\al+1]$. The rest of the proof is as before.
\qedhere
\end{enumerate}
\end{proof}

The following \lcnamecref{lem:derivative} shows that the directional derivative of the Laplacian extends to spaces of lower regularity than predicted by \cref{thm:bochner}.
This is used in \cref{lem:adjoint,thm:satisfies} below.

\begin{lemma}[Derivative of the Laplacian with respect to the metric]
\label{lem:derivative}
Let $\al\in (m/2,\infty)$ with $\al\geq 1$,
let $E$ be a natural first order vector bundle over $M$, 
let $r\in[2-\al,\al]$, 
and let $s \in [2-r,\al]$. 
Then the directional derivative of the Laplace operator with respect to the metric extends to a real analytic mapping
\begin{equation*} 
\Met_{H^{\al}}(M) \times \Ga_{H^r}(S^2T^*M) \ni (g,q) \mapsto D_{g,q}\De^g \in L(\Ga_{H^s}(E), \Ga_{H^{r+s-2-\al}}(E)).
\end{equation*}
If the connection on $E$ does not depend on $g$, then this extends to all $r \in [1-\alpha,\alpha]$ and $s \in [2-r,\alpha+1]$.
\end{lemma}

\begin{proof}
Recall from the proof of \cref{lem:covariant} that the covariant derivative can be expressed as $\nabla^g = \widehat\nabla + A^g$.
In the case $E=TM$ the tensor field $A^g$ is a section of $T^*M\otimes L(TM,TM)$, which is given in abstract index notation by
\begin{align*}
(A^g)^i_{jk}=\frac12g^{il}\big((\widehat\nabla g)_{ljk}+(\widehat\nabla g)_{jkl}-(\widehat\nabla g)_{ljk}\big).
\end{align*}
Thus, the directional derivative $D_{g,q}\nabla^g=D_{g,q}A^g$ satisfies
\begin{align*}
D_{g,q}(A^g)^i_{jk}
=&\tfrac12(D_{g,q}g^{il})\big(({\widehat\nabla}g)_{ljk}+({\widehat\nabla}g)_{jkl}-({\widehat\nabla}g)_{ljk}\big) 
\\&
+\tfrac12g^{il}\big(D_{g,q}({\widehat\nabla}g)_{ljk}+({\widehat\nabla}g)_{jkl}-({\widehat\nabla}g)_{ljk}\big)
\\&
+\tfrac12g^{il}\big(({\widehat\nabla}g)_{ljk}+D_{g,q}({\widehat\nabla}g)_{jkl}-({\widehat\nabla}g)_{ljk}\big)
\\&
+\tfrac12g^{il}\big(({\widehat\nabla}g)_{ljk}+({\widehat\nabla}g)_{jkl}-D_{g,q}({\widehat\nabla}g)_{ljk}\big).
\end{align*}
It follows from the module properties of \cref{thm:module} together with the formulas 
\begin{equation*}
D_{g,q}g\i=-g\i q g\i, 
\qquad
D_{g,q}{\widehat\nabla}g={\widehat\nabla}q,
\end{equation*}
that $D_{g,q}A^g$ extends to a real analytic map
\begin{align*}
\Met_{H^\al}(M)\times \Ga_{H^r}(S^2T^*M) \ni (g,q) \mapsto D_{g,q}A^g \in \Ga_{H^{r-1}}(T^*M\otimes L(TM,TM)).
\end{align*}
Taking $r\leq \alpha$ and $\alpha>d/2$ for granted, this requires continuity of the multiplication $H^r\times H^{\alpha-1}\to H^{r-1}$.
By the multilinear algebra described in the proof of \cref{lem:covariant}, this generalizes from $E=TM$ to all first order natural bundles $E$, yielding a real analytic map
\begin{align*}
\Met_{H^\al}(M)\times \Ga_{H^r}(S^2T^*M) \ni (g,q) \mapsto D_{g,q}A^g \in \Ga_{H^{r-1}}(T^*M\otimes L(E,E)).
\end{align*}
In particular, we will use that this holds for the given bundle $E$ and for its tensor product with $T^*M$.
For any $h \in \Ga_{H^s}(E)$, differentiating the formula $\De^g h=\Tr^{g\i}(\nabla^g\nabla^gh)$ of \cref{thm:bochner} at $g\in\Met_{H^\al}(M)$ in a smooth direction $q\in\Ga(S^2T^*M)$ yields 
\begin{align*}
D_{g,q}\De^gh 
&= -
D_{g,q}(\Tr^{g^{-1}}\nabla^g\nabla^gh) 
\\&= 
-(D_{g,q}\Tr^{g^{-1}})\nabla^g\nabla^gh 
-\Tr^{g^{-1}}(D_{g,q}A)\nabla^gh
-\Tr^{g^{-1}}\nabla^g(D_{g,q}A)h.
\end{align*}
In the special case where the connection on $E$ does not depend on $g$, the last summand above vanishes.
By the $H^{r-1}$ regularity of $D_{g,q}A$ and the module properties of \cref{thm:module}, this formula extends real analytically to all $q \in \Ga_{H^r}(S^2T^*M)$. 
Taking $r\leq \alpha$ and $\alpha>d/2$ for granted, the first summand requires continuity of the multiplication $H^r\times H^{s-2}\to H^{r+s-2-\alpha}$, the second summand requires continuity of the multiplication $H^{r-1}\times H^{s-1}\to H^{r+s-2-\alpha}$, and the third summand requires continuity of the multiplication $H^{r-1}\times H^s\to H^{r+s-1-\alpha}$ and $\nabla\colon H^{r+s-1-\alpha}\to H^{r+s-2-\alpha}$. The third summand requires the additional conditions $s\leq \alpha$ and $r\geq 2-\alpha$, which are not needed anywhere else.
As $h \in \Ga_{H^s}(E)$ was arbitrary, the \lcnamecref{lem:derivative} follows from the real analytic uniform boundedness theorem~\cite[Theorem~11.14]{KM97} . 
\end{proof}

\section{Perturbative spectral theory of sectorial operators}
\label{sec:sectorial}

The main result of this section (\cref{thm:smooth}) is that the functional calculus $A\mapsto f(A)$ is holomorphic for certain classes of operators $A$ and holomorphic functions $f$.
The result hinges on the theory of operators with bounded $\mathcal H^\infty$ calculus and their perturbations \cite{kalton2006perturbation}. 

\subsection{Sectorial operators}
\label{sec:sectorial2}
For each $\om\in[0,\pi]$, the sector $S_\om$ of angle $\pm\om$ about the positive real axis in the complex plane is defined as
\begin{equation*}
S_\om :=\begin{cases} \{ z\in \mathbb C: z\ne 0 \text{ and  }|\on{arg}(z)| <\om\} &\quad \text{ if }\om\in(0,\pi]
\\
(0,\infty) &\quad \text{ if }\om=0.
\end{cases}
\end{equation*}
For any $\om\in(0,\pi]$, $\mathcal H^\infty(S_\om)$ denotes the Banach algebra of bounded holomorphic functions on $S_\om$ with the supremum norm.

Let $A$ be a (possibly unbounded) closed linear operator on a Banach space $X$. 
We denote its resolvent set by $\rho(A)$ and its resolvent by $R_\la(A)=(A-\la)\i$ for $\la\in\rho(A)$.
Then $A$ is called \emph{sectorial} of angle $\om\in [0,\pi)$ if the spectrum of $A$ is contained in $\overline{S_\om}$ and for all $\om'\in(\om,\pi)$, the 
resolvent operators $\la R_\la(A)$ for $\la \in \mathbb C\setminus \overline{S_{\om'}}$ are uniformly bounded in $L(X)$.

Sectorial operators admit a holomorphic functional calculus  as described below, following  \cite[Section~2.5.1]{Haase2006}. 
For simplicity we restrict ourselves to invertible operators, as this eliminates the need for regularization at the origin.
Generalizations to injective operators with dense range have been considered in~\cite{kalton2006perturbation}, and we believe that the results of this section generalize to this setting. In our applications, Riemannian metrics on spaces of mappings, invertibility is a natural assumption and is needed to  formulate  the geodesic equation, see~\cref{sec:conditions} and thus we consider only this simpler situation in this article.

Let $0<\om<\ph<\pi$,
let $r>0$,
let $A$ be an invertible sectorial operator of angle strictly less than $\om$, 
let $\bigcirc$ be a closed ball centered zero contained in $\rho(A)$, 
and let $f$ be a holomorphic function on $S_\ph$ satisfying
\begin{equation*}
\sup_{\la \in S_\ph\setminus\bigcirc}|\la^r f(\la)|<\infty.
\end{equation*}
Then the following Bochner integral is well-defined by the sectoriality of $A$:
\begin{equation*}
f(A) := \frac{-1}{2\pi i}\int_{\partial(S_{\om}\setminus\bigcirc)} f(\la) R_\la(A) d\la \in L(X).
\end{equation*}
This primary functional calculus can be extended to larger classes of functions as described in \cite[Chapter~1]{Haase2006}.
In particular, for any $z \in \mathbb C$, the fractional power $A^z$ is well-defined as an invertible sectorial operator \cite[Chapter~3]{Haase2006}. 
The (homogenous) fractional domain space $\dot X_r$ of $A$ is defined for any $r \in \mathbb R$ as the completion of the domain of $A^r$ with respect to the norm $\|x\|_{\dot X_r} := \|A^r x\|_X$. Note that $\dot X_0=X$. 
Let $\dot X_{<r}=\bigcap_{s<r}\dot X_s$ and $\dot X_{>r}=\bigcup_{s>r}\dot X_s$.
Note, that these spaces coincide with the inhomogenous fractional domain space $X_r$, defined via the graph norm, thanks to the assumed invertibility of the operator $A$.

The following \lcnamecref{lem:holomorphic} shows that the resolvent map is holomorphic in the operator. 
This is the basis for all subsequent considerations. 

\begin{lemma}[Perturbations of sectorial operators]
\label{lem:holomorphic}
Let $A$ be an invertible sectorial operator of angle strictly less than $\om\in(0,\pi)$ on a complex Banach space $X$, 
let $(\dot X_r)_{r\in\mathbb R}$ be the fractional domain spaces associated to $A$, 
and let $\bigcirc$ be a closed centered ball contained in the resolvent set of $A$.
Then there exists an open neighborhood $U$ of $A$ for the $L(\dot X_1,\dot X_0)$-topology such that the following statements hold for all $r \in (-\infty,1]$, $\ph \in (\om,\pi)$, and holomorphic functions $f\colon S_\ph\to\mathbb C$ with $\sup_{\la \in S_\ph\setminus\bigcirc}|\la^r f(\la)|<\infty$.
\begin{enumerate}
\item 
\label{lem:holomorphic:a}
All operators in $U$ are sectorial of angle strictly less than $\omega$, and their resolvent sets contain the ball $\bigcirc$.

\item 
\label{lem:holomorphic:b}
The following map is well-defined and holomorphic:
\begin{align*}
U \ni B &\mapsto (\la\mapsto\la^{1-r} R_\la(B)) \in  C_b(\partial(S_{\om}\setminus\bigcirc),L(\dot X_0,\dot X_r)),
\end{align*}
where $C_b$ denotes the space of globally bounded continuous functions.

\item 
\label{lem:holomorphic:c}
The following map is well-defined and holomorphic:
\begin{align*}
U \ni B &\mapsto (\la\mapsto\la^{1-r} R_\la(B)) \in  C_b(\partial(S_{\om}\setminus\bigcirc),L(\dot X_{1-r},\dot X_1)).
\end{align*}

\item 
\label{lem:holomorphic:d}
Assume that $A$ is densely defined, let $\mathbb D$ be the open unit ball in $\mathbb C$,
and let $B\colon\mathbb D\to U$ be a holomorphic map such that $\sup_{z \in \mathbb D} \|f(B(z))\|_{L(\dot X_0,\dot X_r)}<\infty$. 
Then the following map is holomorphic:
\begin{equation*}
\mathbb D \ni z \mapsto f(B(z)) = \frac{-1}{2\pi i}\int_{\partial(S_{\om}\setminus\bigcirc)} f(\la) R_\la(B(z)) d\la \in L(\dot X_0,\dot X_r),
\end{equation*} 
where the integral converges in $L(\dot X_0,\dot X_{<r})$ and $L(\dot X_{>1-r},\dot X_1)$.
\end{enumerate}
\end{lemma}

Note that \cref{lem:holomorphic}.\ref{lem:holomorphic:b} implies that $B\mapsto f(B)$ is holomorphic with values in $L(\dot X_0,\dot X_{<r})$. 
Similarly, \cref{lem:holomorphic}.\ref{lem:holomorphic:c} implies that $B\mapsto f(B)$ is holomorphic with values in $L(\dot X_{>1-r},\dot X_1)$. 
In either case there is a loss of regularity. 
Point \ref{lem:holomorphic:d} shows that this loss of regularity can be overcome using bounds on the functional calculus.
Indeed, it implies that $f(B)$ is holomorphic in $B$ for any convenient operator topology such that $B\mapsto\|f(B)\|_{L(\dot X_0,\dot X_r)}$ is locally bounded; see \cref{sec:convenient}.
This will be exploited in \cref{lem:smooth} below.

\begin{proof}
\begin{enumerate}[wide]
\item 
Choose $\om'$ strictly greater than the angle of sectoriality of $A$ and strictly smaller than $\pi$, 
fix a centered closed ball $\bigcirc$ in the resolvent set of $A$,  
let $\La=\mathbb C \setminus S_{\om'} \cup \bigcirc$, 
and define constants $a,b \in (0,\infty)$ by
\begin{align*}
a\i &= 3\sup_{\la\in \La}\|R_\la(A)\|_{L(\dot X_0)}<\infty,
\\
b\i &= 3\sup_{\la\in \La}\|A R_\la(A)\|_{L(\dot X_0)} 
= 3\sup_{\la\in \La} \|\on{Id}_X+\la R_\la(A)\|_{L(\dot X_0)}
\\&
\le 3(1+\sup_{\la\in \La} \|\la R_\la(A)\|_{L(\dot X_0)})<\infty.
\end{align*}
Here the bounds for small $\la$ follow from the invertibility and for large $\la$ from the sectoriality of $A$.
Let $U$ be the set of all $B \in L(\dot X_1,\dot X_0)$ with $\|B-A\|_{L(\dot X_1,\dot X_0)}< b$. 
Then the definitions of $a$ and $b$ imply for all $\la\in\La$ that
\begin{equation*}
a\|R_\la(A)\|_{L(\dot X_0)} + b \|A R_\la(A)\|_{L(\dot X_0)}\le \frac23 <1,
\end{equation*}
and the definition of $U$ implies for all $B \in U$ and $x\in \dot X_1$ that
\begin{equation*}
\| (B-A)x\|_{\dot X_0} \le \| B-A\|_{L(\dot X_1,\dot X_0)} \|x\|_{\dot X_1} \le b \| Ax\|_{\dot X_0}.
\end{equation*}
By \cite[Theorem~IV.3.17]{Kato76} these estimates show that $\La$ is contained in the resolvent set of $B$, and the resolvent of $B$ satisfies for all $\la\in\La$ that
\begin{equation*}
\| R_\la(B) \|_{L(\dot X_0)} \le \frac{\|R_\la (A)\|_{L(\dot X_0)}}{1-a\| R_\la(A)\|_{L(\dot X_0)} - b \| AR_\la (A)\|_{L(\dot X_0)} } 
\le 3 \|R_\la (A)\|_{L(\dot X_0)}.
\end{equation*}
Hence, $B$ is sectorial of angle $\omega'$ on $\dot X_0$, and the resolvent set of $B$ contains $\bigcirc$.

\item 
For each $\la \in \La$, the resolvent $(U\ni B \mapsto R_\la(B)\in L(\dot X_0,\dot X_{1}))$ is holomorphic. 
As $U$ is a metric ball, the following series converges in $L(\dot X_0,\dot X_{1})$ for all $B$ in this ball: 
\begin{align*}
R_\lambda(B)
&=
\sum_{n\in\mathbb N} \frac{R^{(n)}_\la(A)(B-A)^n}{n!}
=
\sum_{n\in\mathbb N} R_\la(A) \big( (B-A)R_\la(A)\big)^n,
\end{align*}
where the second equality can be verified easily by induction on $n$.
We need some resolvent estimates to show that this series converges uniformly in $\la$ in appropriate topologies. For all $r \in (-\infty,1)$, one has
\begin{align*}
\sup_{\la\in\partial S_{\om'}} \|\la^{1-r}A^rR_\la(A)\|_{L(\dot X_0)} 
=
\sup_{\nu=e^{\pm i\om'}}\sup_{t\in\mathbb R_{>0}} \|\psi_\nu(tA)\|_{L(\dot X_0)} < \infty,
\end{align*}
where $\psi_\nu(z)=\nu^{1-r}z^r(z-\nu)\i$ and the bound follows from 
\begin{align*}
\psi_\nu(tA) 
&=
\int_{\partial(S_\om\setminus\bigcirc)}\psi_\nu(\la) R_\la(tA)d\la
=
\int_{\partial(S_\om\setminus\bigcirc)}\la\i\psi_\nu(\la)\cdot
\tfrac{\la}{t} R_{\frac{\la}{t}}(A)d\la,
\end{align*}
where under the integral on the right-hand side the first factor is integrable and the second factor is bounded.
Together with the bounds in \ref{lem:holomorphic:a} this shows for all $r \in (-\infty,1]$ that
\begin{equation*}
\sup_{\la\in\La} \|\la^{1-r}A^rR_\la(A)\|_{L(\dot X_0)} < \infty.
\end{equation*}
Therefore, one has for all $r \in (-\infty,1]$ that
\begin{align*}
\hspace{2em}&\hspace{-2em}
\sum_{n\in\mathbb N}\sup_{\la\in\La} \Big\lVert \la^{1-r}A^r R_\la(A) \big( (B-A)R_\la(A) \big)^n\Big\rVert_{L(\dot X_0)}
\\&\leq
\sum_{n\in\mathbb N} \sup_{\la,\mu\in\La} \lVert \la^{1-r}A^r R_\la(A) \rVert_{L(\dot X_0)} \lVert B-A\rVert_{L(\dot X_{1},\dot X_0)}^n \lVert R_\mu(A)\rVert_{L(\dot X_0,\dot X_{1})}^n.
\end{align*}
By the definition of $b$, the right-hand side is finite if
\begin{equation*}
\lVert B-A\rVert_{L(\dot X_{1},\dot X_0)} < 3b,
\end{equation*}
which holds true for all $B \in U$.
This proves \ref{lem:holomorphic:b}.

\item can be shown as in \ref{lem:holomorphic:b} using the estimate
\begin{align*}
\hspace{2em}&\hspace{-2em}
\sum_{n\in\mathbb N}\sup_{\la\in\La} \Big\lVert  R_\la(A) \big( (B-A)R_\la(A) \big)^n \la^{1-r}A^r\Big\rVert_{L(\dot X_{1})}
\\&\leq
\sum_{n\in\mathbb N} \sup_{\la,\mu\in\La} \lVert R_\mu(A)\rVert_{L(\dot X_0,\dot X_{1})}^n \lVert B-A\rVert_{L(\dot X_{1},\dot X_0)}^n \lVert R_\la(A) \la^{1-r}A^r  \rVert_{L(\dot X_{1})}.
\end{align*}

\item 
Let $s < r$,
and let $x \in \dot X_0$.
As $A$ is densely defined, there is a sequence $(x_n)_{n\in\mathbb N}$ in $\dot X_{1-s}$ which converges to $x$ in the $\dot X_0$ topology.  
By \ref{lem:holomorphic:c} and the continuous inclusion of $\dot X_{1}$ in $\dot X_{r}$, the following map is holomorphic for any $n\in\mathbb N$:
\begin{equation*}
U \ni B \mapsto (\la\mapsto\la^{1-s} R_\la(B)x_n) \in C_b(\partial(S_{\om}\setminus\bigcirc),\dot X_{r}).
\end{equation*}
This implies that the map
\begin{equation*}
U \ni B \mapsto f(B)x_n = \frac{-1}{2\pi i}\int_{\partial(S_{\om}\setminus\bigcirc)} f(\la)\la^{s-1} \cdot \la^{1-s}R_\la(B)x_n d\la \in \dot X_{r}
\end{equation*}
is holomorphic, where under the integral the first factor is integrable and the second one bounded. 
By Cauchy's integral theorem, one obtains for any closed ball $D \subset \mathbb D$ and any $z$ in the interior of $D$ that
\begin{equation*}
f(B(z))x_n = \frac{-1}{2\pi i} \int_{\partial D} \frac{f(B(w))x_n}{z-w} dw \in \dot X_{r}.
\end{equation*}
The assumption $\sup_{w \in \mathbb D} \|f(B(w))\|_{L(\dot X_0,\dot X_{r})}<\infty$ allows one to take the limit $n\to\infty$, which shows that
\begin{equation*}
f(B(z))x = \frac{-1}{2\pi i} \int_{\partial D} \frac{f(B(w))x}{z-w} dw \in \dot X_{r}.
\end{equation*}
This shows that $z\mapsto f(B(z))x$ is holomorphic. 
As this holds for all $x \in \dot X_0$, one obtains from the holomorphic uniform boundedness theorem that $z\mapsto f(B(z))$ is holomorphic, as claimed. 
The resolvent integrals converge in $L(\dot X_0,\dot X_{<r})$ and $L(\dot X_{>1-r},\dot X_1)$ thanks to \ref{lem:holomorphic:b}--\ref{lem:holomorphic:c}. 
This concludes the proof of \ref{lem:holomorphic:d}.
\qedhere
\end{enumerate}
\end{proof}

\subsection{Bounded \texorpdfstring{$\mathcal H^\infty$}{H infinity} calculus and R-sectoriality}

Let $A$ be an invertible sectorial operator of positive angle strictly less than $\om\in(0,\pi)$ on a complex Banach space $X$.
Then each bounded holomorphic function $f$ on $S_\om$ defines a possibly unbounded closed linear operator $f(A)$ \cite[Section~2.5.1]{Haase2006}.
The operator $A$ is said to admit a \emph{bounded $\mathcal H^\infty(S_\om)$ calculus} if \cite[Section~5.3]{Haase2006}
\begin{equation*}
\sup_{f\in\mathcal H^\infty(S_\om)\setminus\{0\}} \frac{\|f(A)\|_{L(X)}}{\|f\|_{\mathcal H^\infty(S_\om)}} < \infty,
\end{equation*}
where $\|\cdot\|_{\mathcal H^\infty(S_\om)}$ is the supremum norm of bounded holomorphic functions on $S_\om$.

We will use in \cref{lem:smooth} below that boundedness of the $\mathcal H^\infty$ calculus is stable under perturbations which are relatively bounded in two distinct fractional domain scales \cite[Theorem~6.1]{kalton2006perturbation}. 
This has been proven first by J.~Pr\"uss (1994) in an unpublished article called ``Perturbation theory for the class $\mathcal H^\infty(X)$'' and published first in \cite{denk2004new}.
Moreover, we will use repeatedly that the fractional domain spaces $(\dot X_r)_{r\in\mathbb R}$ associated to operators with bounded $\mathcal H^\infty$ calculus are complex interpolation spaces; see \cite[Proposition~2.2]{kalton2006perturbation} or \cite[Lemma~4.13]{lunardi2018interpolation}.
This characterization is available also for the larger class of operators with bounded imaginary powers, but there are no corresponding perturbative results for this class \cite[Section~4.6]{Arendt2004}.

Boundedness of the $\mathcal H^\infty$ calculus implies a high degree of unconditionality, i.e., norm boundedness can be replaced by R-boundedness in several regards (cf.\@ \cite[Section~5.6]{Haase2006} and \cite[Section~4]{kalton2006perturbation}). This follows from quadratic estimates first developed by~\cite{mcintosh1986operators,cowling1996banach}.

A set $F\subseteq L(X,Y)$ of linear operators between Banach spaces $X$ and $Y$ is called R-bounded \cite[Section~3]{kalton2006perturbation} if there exists a constant $C>0$ such that the following inequality holds for all $n\in\mathbb N$, $x_1,\dots,x_n\in X$, $B_1,\dots,B_n \in F$, and independent Rademacher random variables $\ep_1,\dots,\ep_n$: 
\begin{equation*}
\mathbb E\Big\|\sum_k \ep_k B_k x_k\Big\|_Y^2 \leq C^2 \mathbb E\Big\|\sum_k\ep_k x_k\Big\|_X^2.	
\end{equation*}
A closed linear operator $A$ on a Banach space $X$ is called R-sectorial of angle $\om\in[0,\pi)$ 
if the spectrum of $A$ is contained in $\overline{S_\om}$ and for all $\om'\in(\om,\pi)$, 
the set $\{\la R_\la(A): \la \in \mathbb C\setminus \overline{S_{\om'}}\}\subseteq L(X)$ is R-bounded \cite[Section~3]{kalton2006perturbation}.
On Hilbert spaces the notions of boundedness and R-boundedness coincide \cite[Section~1]{kalton2014hinfinity}.

The following \lcnamecref{lem:smooth} carries out the program hinted at in \cref{lem:holomorphic}.\ref{lem:holomorphic:d}: 
it identifies an operator topology such that boundedness of the $\mathcal H^\infty$ calculus is an open condition and uses the bounds on the $\mathcal H^\infty$ calculus to deduce that the functional calculus is holomorphic without any loss of regularity.

\begin{lemma}[Perturbations of operators with bounded \texorpdfstring{$\mathcal H^\infty$}{H infinity} calculus]
\label{lem:smooth}
Let $A$ be an invertible densely defined R-sectorial operator of positive angle strictly less than $\om\in(0,\pi)$ with bounded $\mathcal H^\infty(S_\om)$ calculus on a complex Banach space $X$, 
let $(\dot X_r)_{r\in\mathbb R}$ be the fractional domain spaces associated to $A$, 
let $\bigcirc$ be a closed centered ball contained in the resolvent set of $A$, 
let $\delta \in \mathbb R\setminus\{0\}$,
and let $V=L(\dot X_1,\dot X_0)\cap L(\dot X_{\de+1},\dot X_\de)$.
Then there exists an open neighborhood $U$ of $A\in V$ such that following statements hold for all $r \in [0,1]$ and $\ph\in(\om,\pi)$.
\begin{enumerate}
\item 
\label{lem:smooth:a}
All operators  $B \in U$ are R-sectorial of positive angle strictly less than $\om$, have resolvent sets which contain the ball $\bigcirc$, and admit a bounded $\mathcal H^\infty(S_\ph)$ calculus with uniform bounds
\begin{equation*}
\sup_{B \in U}\ \sup_{g \in \mathcal H^\infty(S_\ph)\setminus\{0\}} \frac{\|g(B)\|_{L(X)}}{\|g\|_{H^\infty(S_\ph)}}
+ \|B^{-r}\|_{L(\dot X_0,\dot X_r)}
<\infty.
\end{equation*}

\item 
\label{lem:smooth:c}
For any holomorphic function $f\colon S_\ph\to\mathbb C$ with $\sup_{\la \in S_\ph} |\la^r f(\la)|<\infty$,
the following map is well-defined and holomorphic, 
\begin{equation*}
U \ni B \mapsto f(B) = \int_{\partial(S_{\om}\setminus\bigcirc)} f(\la) R_\la(B) d\la \in L(X, \dot X_r),
\end{equation*}
where the integral converges in $L(\dot X_0,\dot X_{<r})$ and $L(\dot X_{>1-r},\dot X_1)$
\end{enumerate}
\end{lemma}

\begin{proof}
\begin{enumerate}[wide]
\item is proven in the three subsequent steps \ref{lem:smooth:a1}--\ref{lem:smooth:a3}. 
\begin{enumerate}[wide, label={\textnormal{\bfseries (a\arabic*)}}]
\item 
\label{lem:smooth:a1}
By \cref{lem:holomorphic} there is a neighborhood $U$ of $A \in V$ such that all operators in $U$ are sectorial of angle strictly less than $\om$, and their resolvent sets contain the ball $\bigcirc$. 

\item 
\label{lem:smooth:a2}
We claim that $U$ may be replaced by a smaller neighborhood of $A$ such that all operators in $U$ are $R$-sectorial and have uniformly bounded $\mathcal H^\infty(S_\ph)$ calculus:
\begin{equation*}
\sup_{B \in U}\ \sup_{g \in \mathcal H^\infty(S_\ph)\setminus\{0\}} \frac{\|g(B)\|_{L(X)}}{\|g\|_{H^\infty(S_\ph)}}
<\infty.
\end{equation*}
As the fractional domain spaces of $A$ are complex interpolation spaces \cite[Proposition~2.2]{kalton2006perturbation}, we may assume $|\de|<1$ for the sake of the subsequent arguments.
In the case $\de<0$ the claim follows from \cite[Theorem~6.1]{kalton2006perturbation} (with the sign of $\de$ reversed) by noting that in this theorem the condition $\on{ran} B \subseteq \on{ran} A^{-\de}$ is not needed as long as $A^{\de}BA^{-\de-1}$ extends to a continuous operator on $X$, which is the case here. 
The $\mathcal H^\infty(S_\ph)$ calculus is bounded uniformly on $U$, as can be seen by tracking the constants in \cite[Theorem~6.1]{kalton2006perturbation}. 
Indeed, in \cite{kalton2006perturbation}, Theorem~6.1 is based on Lemma~6.2, and the constant $C$ in the proof of this lemma is uniform in $B \in U$ because it depends only on $\|A^{\de}BA^{-\de-1}\|_{L(X)}$ and the R-bound of $\{M(\la):|\arg \la|\geq\om\}$, which is again uniform in $B$. 
The constant $C$ of \cite[Lemma~6.2]{kalton2006perturbation} is passed on to Theorem~4.1.(iii), which is proven in Proposition~4.6. 
This proposition uses only the R-sectoriality of $B$ and therefore furnishes uniform bounds. 
This proves the claim in the case $\de<0$. 
In the case $\de>0$ the claim follows from \cite[Corollary~6.5]{kalton2006perturbation} (again with the sign of $\de$ reversed).
This corollary is based on \cite[Theorem~6.1]{kalton2006perturbation} and also furnishes uniform bounds for the $\mathcal H^\infty(S_\ph)$ calculus.
This proves \ref{lem:smooth:a2}. 

\item
\label{lem:smooth:a3}
For $r \in\{0,1\}$ the condition $\sup_{B\in U} \|B^{-r}\|_{L(\dot X_0,\dot X_r)}<\infty$ is trivially satisfied. 
Thus, we restrict to the case $r \in (0,1)$. 
As the operators $B\in U$ have bounded $\mathcal H^\infty(S_\ph)$ calculus, their associated fractional domains $(\dot X_{r,B})_{r\in\mathbb R}$ are complex interpolation spaces \cite[Proposition~2.2]{kalton2006perturbation}.
Thus, there is a constant $C>0$ such that the following estimate holds for all $B \in U$:
\begin{align*}
\|B^{-r}\|_{L(\dot X_0,\dot X_r)}
&=
\|\Id\|_{L(\dot X_{r,B},\dot X_r)}
\leq
C \|\Id\|_{L(\dot X_{0,B},\dot X_0)}^{1-r}\|\Id\|_{L(\dot X_{1,B},\dot X_1)}^r
\\&=
C \|B^{-1}\|_{L(\dot X_0,\dot X_1)}^r.
\end{align*}
The right-hand side is bounded uniformly on $U$. 
This proves \ref{lem:smooth:a3} and concludes the proof of \ref{lem:smooth:a}.
\end{enumerate}

\item  
Let $\mathbb D$ denote the open unit ball in $\mathbb C$, 
and let $B\colon\mathbb D \to U$ be a holomorphic map. 
Then \ref{lem:smooth:a} implies that
\begin{equation*}
\sup_{z \in \mathbb D} \|f(B(z))\|_{L(\dot X_0,\dot X_r)}
\leq 
\sup_{B \in U} \|B^{-r}\|_{L(\dot X_0,\dot X_r)} \|B^rf(B)\|_{L(\dot X_0,\dot X_0)} < \infty.
\end{equation*}
Thus, \cref{lem:holomorphic}.\ref{lem:holomorphic:d} shows that the curve $f(B)\colon\mathbb D \to L(X,\dot X_r)$ is holomorphic.
By convenient calculus, this implies \ref{lem:smooth:c}.
\qedhere
\end{enumerate} 
\end{proof}

The following theorem, which is the main result of this section, sums up some implications of \cref{lem:smooth} in the common situation where the perturbations can be controlled in the $L(\dot X_{\al+1},\dot X_\al)$ topology for all $\al$ in an interval $[\beta,\gamma]$ or equivalently for all $\al\in\{\be,\ga\}$ by complex interpolation.

\begin{theorem}[Perturbations of operators with bounded \texorpdfstring{$\mathcal H^\infty$}{H infinity} calculus]
\label{thm:smooth}
Let $A$ be an invertible densely defined R-sectorial operator of positive angle strictly less than $\om\in(0,\pi)$ with bounded $\mathcal H^\infty(S_\om)$ calculus on a complex Banach space $X$, 
let $(\dot X_r)_{r\in\mathbb R}$ be the fractional domain spaces associated to $A$,
let $\be,\ga \in\mathbb R$ with $\be<\ga$, 
and let $V=L(\dot X_{\be+1},\dot X_\be) \cap L(\dot X_{\ga+1},\dot X_\ga)$.
Then there exists an open neighborhood $U$ of $A\in V$ such that 
for all $r,s\in\mathbb R$ with $s,s+r\in[\be,\ga+1]$,
$\ph\in(\om,\pi)$, 
and holomorphic functions $f\colon S_\ph\to\mathbb C$ with $\sup_{\la\in S_\ph}|\la^r f(\la)|<\infty$, 
the following map is well-defined and holomorphic:
\begin{equation*}
U \ni B \mapsto f(B) \in L(\dot X_s,\dot X_{s+r}).
\end{equation*}
\end{theorem}

\begin{proof} 
As $A$ has bounded $H^\infty(S_\om)$ calculus, the fractional domain spaces are complex interpolation spaces  \cite[Proposition~2.2]{kalton2006perturbation}.
Thus, $V$ is continuously embedded in all intermediate spaces $L(\dot X_{\al+1},\dot X_\al)$ with $\al\in(\be,\ga)$. 
We first focus on the case $r\geq 0$. 
For each $s \in [\be,\ga]$ and $r \in [0,1]$ the statement follows from \cref{lem:smooth} applied to the operator $A \in L(\dot X_{s+1},\dot X_s)$.
The conditions of the \lcnamecref{lem:smooth} are satisfied because there is always space below or above the interval $[s,s+1]$ within the larger interval $[\be,\ga+1]$.
The statement can be extended to higher values of $r$ by composition with integer powers of $B$.
This shows that the statement holds for all $s \in [\be,\ga]$ and nonnegative $r$ with $s+r \leq \ga+1$. 
The remaining case where $s \in (\ga,\ga+1]$ and $r$ is nonnegative with $s+r \leq \ga+1$ is covered by writing 
\begin{equation*}
f(B)\colon \dot X_s \xrightarrow{(B^{\ga-s})\i} \dot X_\ga \xrightarrow{B^{\ga-s} f(B)}\dot X_r,	
\end{equation*}
where the first arrow is holomorphic in $B$ because inversion is holomorphic, and the second arrow is holomorphic in $B$ thanks to \cref{lem:smooth} applied to the function $\la\mapsto \la^{\ga-s}f(\la)$ and the operator $A \in L(\dot X_{\ga+1},\dot X_\ga)$. 
Thus, we have shown the statement for all $r\geq 0$. The corresponding statement for $r\leq 0$ can be obtained by writing
\begin{equation*}
f(B)\colon \dot X_s \xrightarrow{B^r f(B)} \dot X_s \xrightarrow{(B^r)\i}\dot X_r,	
\end{equation*}
where the first arrow is holomorphic in $B$ thanks to \cref{lem:smooth} applied to the function $\la\mapsto \la^r f(\la)$, and the second arrow is holomorphic in $B$ because inversion is holomorphic. 
\end{proof}

\begin{remark}[Real Banach spaces]
The results in \cref{lem:holomorphic,lem:smooth} generalize to real Banach spaces $X$ as follows.
The resolvent mappings in \cref{lem:holomorphic}.\ref{lem:holomorphic:b}--\ref{lem:holomorphic:c} are real analytic because real analytic mappings between real Banach spaces extend to holomorphic mappings on small complex neighborhoods. 
This implies real analyticity of the resolvent integrals in $L(\dot X_0,\dot X_{<r})$ and $L(\dot X_{>1-r},\dot X_1)$, i.e., with a loss of regularity.
\cref{lem:holomorphic}.\ref{lem:holomorphic:d}, where there is no loss of regularity, generalizes from holomorphic to real analytic curves provided that the bound on $f(B(z))$ holds not only for real $z$, but also for nearby $z$ with small imaginary part. 
This can be difficult to verify if the holomorphic extension is not given explicitly. 
This problem is settled in \cref{lem:smooth}, which implies for real Banach spaces $X$ and $V$ that the functional calculus $B\mapsto f(B)$ is real analytic.
\end{remark}

\section{Perturbative spectral theory of Laplace operators}
\label{sec:perturbative}

In this section the perturbative spectral theory of \cref{sec:sectorial} is applied to the particular case of Laplace operators on compact Riemannian manifolds. 
The perturbations are taken with respect to the Riemannian metric in Sobolev topologies.
We first present some auxiliary results about functional calculus and fractional domain spaces of Laplace operators (\cref{lem:functional,lem:fractional}) and then prove our main result (\cref{thm:perturbations}) on perturbations of fractional Laplacians. 

The following \lcnamecref{lem:functional} describes the functional calculus of the Laplace operator associated to a fixed metric. 
The Laplacian is considered as an operator from $\Ga_{H^1}(E)$ to $\Ga_{H^{-1}}(E)$ because this is the only option which works simultaneously for all Sobolev regularities $\al\geq1$ of the metric.

\begin{lemma}[Functional calculus of Laplacians]
\label{lem:functional}
Let $\al\in (m/2,\infty)$ with $\alpha\geq 1$, 
let $g \in \Met_{H^\al}(M)$, 
and let $E$ be a natural first order vector bundle over $M$.
Then the following statements hold: 
\begin{enumerate}
\item 
\label{lem:functional:a}
The operator $1+\De^g\colon \Ga_{H^1}(E)\to\Ga_{H^{-1}}(E)$ is invertible, and the following bilinear form is an equivalent scalar product on $\Gamma_{H^{-1}}(E)$:
\begin{equation*}
\Ga_{H^{-1}}(E)\times \Ga_{H^{-1}}(E) \ni (h,k) \mapsto \langle (1+\De^g)\i h,k\rangle_{H^0(g)} \in \mathbb R,
\end{equation*}
where $\langle\cdot,\cdot\rangle_{H^0(g)}\colon \Ga_{H^1}(E)\times\Ga_{H^{-1}}(E)\to\mathbb R$ is the $H^0(g)$ duality of \cref{lem:volume}.
We will write $\Ga_{H^{-1}(g)}(E)$ for the space $\Ga_{H^{-1}}(E)$ with this scalar product. 

\item 
\label{lem:functional:b}
The operator $1+\De^g$ with domain $\Ga_{H^1}(E)$ is unbounded self-adjoint on the space $\Ga_{H^{-1}(g)}(E)$ and has a compact inverse. Thus, there exists an orthonormal basis of eigenvectors $(e_i)_{i\in\mathbb N}$ in $\Ga_{H^{-1}(g)}(E)$ and a non-decreasing sequence of eigenvalues $(\la_i)_{i\in\mathbb N}$ in $(0,\infty)$ such that
\begin{equation*}
\forall i \in \mathbb N: 
\quad
e_i \in \Ga_{H^1}(E), 
\quad
(1+\Delta^g) e_i = \lambda_i e_i.
\end{equation*}

\item
\label{lem:functional:c}
For each function $f\colon \{\lambda_1,\lambda_2,\dots\} \to \mathbb R$ the following is a densely defined self-adjoint linear operator on $\Ga_{ H^{-1}(g)}(E)$:
\begin{equation*}
f(1+\De^g)\colon
\left\{\begin{aligned}
\on{Dom}(f(1+\Delta^g)) &\to \Ga_{ H^{-1}}(E), 
\\
h &\mapsto \sum_{i\in\mathbb N}\langle h_i,e_i\rangle_{\Ga_{ H^{-1}(g)}(E)} f(\lambda_i) e_i,
\end{aligned}\right.\end{equation*}
where
\begin{equation*}
\on{Dom}(f(1+\Delta^g)) = \left\{h \in \Ga_{H^{-1}(g)}(E); \sum_{i\in\mathbb N} \langle h_i,e_i\rangle^2_{\Ga_{ H^{-1}(g)}(E)} f(\lambda_i)^2<\infty\right\}.
\end{equation*}

\item 
\label{lem:functional:d}
Let $\ph \in (0,\pi)$, 
recall that $S_\ph := \{z \in \mathbb C\setminus\{0\}: |\arg z|<\ph\}$ denotes the open sector of angle $\pm\ph$ about the positive real axis,
and let $f\colon S_\ph\to\mathbb C$ be a holomorphic function which satisfies for some $r \in (0,\infty)$ that $
\sup_{\lambda \in S_\ph}|\lambda^r f(\lambda)|<\infty$.
Then the operator $f(1+\De^g) \in L(\Ga_{ H^{-1}(g)}(E))$ coincides with the following resolvent integral for any $\om\in(0,\ph)$ and any closed centered ball $\bigcirc$ in the resolvent set of $1+\De^g$:
\begin{equation*}
f(1+\Delta^g)=-\frac1{2\pi i}\int_{\partial(S_{\om}\setminus\bigcirc)} f(\lambda)R_\la(1+\Delta^g) d\lambda \in L(\Ga_{H^{-1}(g)}(E)),
\end{equation*}
where the integral converges in $L(\Ga_{H^{-1}(g)}(E))$. 
\end{enumerate}
\end{lemma} 

\begin{proof}
\begin{enumerate}[wide]
\item 
The operator $\De^g$ is non-negative symmetric with respect to the $H^0(g)$ pairing $\langle\cdot,\cdot\rangle_{H^0(g)}$ of \cref{lem:volume}.\ref{lem:volume:b} in the following sense: 
\begin{equation*}
\forall h,k\in\Ga_{H^1}(E):\qquad \langle \De^g h,k\rangle_{H^0(g)} = \langle h,\De^g k\rangle_{H^0(g)} ,\qquad \langle \De^g h,h\rangle_{H^0(g)} \geq 0.
\end{equation*}
This can be seen via approximation by smooth $g,h,k$ using the continuity of $g\mapsto\langle\cdot,\cdot\rangle_{H^0(g)}$ established in \cref{lem:volume} and the continuity of $g\mapsto \De^g$ established in \cref{thm:bochner}.
This implies that $1+\De^g$ is strictly positive and thus injective. 
As an injective operator it is semi-Fredholm, which implies that its index is well-defined in $[-\infty,\infty)$. 
It actually has vanishing index because the curve $t\mapsto t+\De^g$ deforms it continuously into the Laplace operator, which has vanishing index by \cref{thm:bochner}.\ref{thm:bochner:a}.
Thus, $1+\De^g\colon \Ga_{H^1}(E)\to\Ga_{H^{-1}}(E)$ is continuously invertible. 
Therefore, the bilinear form in \ref{lem:functional:a} is a weak scalar product, which we denote by $\langle\cdot,\cdot\rangle_{\Ga_{H^{-1}(g)}(E)}$.
For any two metrics $g,\hat g \in \Met_{H^\al}(M)$, let 
\begin{equation*}
I^{g,\hat g} := \left(h\mapsto (1+\De^{\hat g}) \frac{\vol^g}{\vol^{\hat g}} \hat g\i g (1+\De^g)\i h \right) \in L(\Ga_{H^{-1}}(E)).
\end{equation*}
Then $I^{g,\hat g}$ is continuous with continuous inverse $I^{\hat g,g}$, the scalar products induced by $g$ and $\hat g$ are related by
\begin{equation*}
\langle h,k\rangle_{\Ga_{H^{-1}(g)}(E)} = \langle I^{g,\hat g} h,k\rangle_{\Ga_{H^{-1}(\hat g)}(E)},
\end{equation*}
and the norms induced by $g$ and $\hat g$ are related by
\begin{align*}
\|\Id\|_{L(\Ga_{H^{-1}(\hat g)}(E),\Ga_{H^{-1}(g)}(E))}
&\leq
\|I^{g,\hat g}\|^{1/2}_{L(\Ga_{H^{-1}(\hat g)}(E))},
\\
\|\Id\|_{L(\Ga_{H^{-1}(g)}(E),\Ga_{H^{-1}(\hat g)}(E))}
&\leq 
\|I^{g,\hat g}\|^{1/2}_{L(\Ga_{H^{-1}(\hat g)}(E))}
\|I^{\hat g,g}\|_{L(\Ga_{H^{-1}(\hat g)}(E))}.
\end{align*}
Thus, these norms are equivalent. 
Moreover, it is well-known that the norm induced by any smooth metric $\hat g \in \Met(M)$ is equivalent to the norm on $\Ga_{H^{-1}}(E)$. 
This concludes the proof of \ref{lem:functional:a}.

\item 
The operator $1+\De^g\colon \Ga_{H^1}(E)\to\Ga_{H^{-1}}(E)$ is symmetric with respect to the $H^0(g)$ pairing and invertible by \ref{lem:functional:a}.
This implies that its inverse is symmetric with respect to the scalar product on $\Ga_{H^{-1}(g)}(E)$ and everywhere defined, thus self-adjoint. 
Therefore, also the operator $1+\De^g$ is self-adjoint as an unbounded linear operator on $\Ga_{ H^{-1}(g)}(E)$.
Its inverse is a compact operator because $\Ga_{H^1}(E)$ is compactly embedded in $\Ga_{H^{-1}}(E)$.
Thus, the spectral properties of compact self-adjoint positive operators imply \ref{lem:functional:b}. 

\item follows from the well-known functional calculus for unbounded self-adjoint linear operators; see e.g.\@ \cite[Theorem~VII.3.2]{werner2005funktionalanalysis}. 

\item is the holomorphic functional calculus for invertible sectorial operators described in \cref{sec:sectorial}. 
The operator $1+\De^g$ is invertible by \ref{lem:functional:b}.
Its eigenvalues are contained in $[1,\infty)$, and the norm of its resolvent can be estimated by the distance to the closest eigenvalue: for any $\om\in(0,\pi)$,
\begin{align*}
\sup_{\lambda\in\mathbb C \setminus \overline{S_\om}} \lVert\lambda R_\la(1+\Delta^g)\rVert_{L(\Ga_{H\i}(E))} 
\leq
\sup_{\lambda\in \partial S_{\om}} \frac{|\lambda|}{\dist(\lambda,[1,\infty))}
< \infty.
\end{align*}
Thus, $1+\De^g$ is sectorial of angle zero.
Together with the assumed decay of $f$ this yields
\begin{equation*}
\int_{\partial(S_{\om}\setminus\bigcirc)} |f(\la)|\ \lVert R_\la(1+\Delta^g)\rVert_{L(\Ga_{H\i}(E))} d\lambda <\infty.
\end{equation*}
This shows convergence of the resolvent integral in $L(\Ga_{H\i}(E))$.
The holomorphic functional calculus coincides with the one described in \ref{lem:functional:c} thanks to Cauchy's residual theorem because the region $S_{\om}\setminus\bigcirc$ contains all eigenvalues of $1+\De^g$ in its interior, the resolvent is holomorphic away from the eigenvalues, and the residuals of the resolvent at the eigenvalues are projections onto the corresponding eigenspaces. 
\qedhere
\end{enumerate}
\end{proof}

\subsection{Fractional domain spaces}
\label{sec:domain}
Let $\al\in(m/2,\infty)$ with $\al\geq 1$,
let $g \in \Met_{H^\al}(M)$,
let $E$ be a natural first order vector bundle over $M$,  
let $A$ be the self-adjoint positive linear operator $1+\De^g$ on $\Ga_{H^{-1}}(E)$ with $\on{Dom}(A)=\Ga_{H^1}(E)$, 
and let $(\dot X_r)_{r\in\mathbb R}$ be the fractional domain spaces of $A$ (see \cref{sec:sectorial2}). 
Note that these spaces are Hilbert spaces \cite[Theorem~4.36]{lunardi2018interpolation}, which coincide with the Bessel potential spaces for smooth Riemannian metrics $g$.
For any $r \in \mathbb R$, we define  
\begin{equation*}
\Ga_{H^r(g)}(E):=\dot X_{(r+1)/2}
\end{equation*}
with equality of norms.  
This notation is justified by \cref{lem:fractional}.\ref{lem:fractional:b} below, which establishes an isomorphism between $\Ga_{H^r(g)}(E)$ and $\Ga_{H^r}(E)$ for certain values of $r$. 
It should be kept in mind, however, that these spaces are in general not isomorphic for other values of $r$ and never isometric. 

Note the shift in the scales of spaces $\Ga_{H^r(g)}(E)$ and $\dot X_{(r+1)/2}$. 
This shift comes from the fact that $A$ is the Laplacian on $\Ga_{H^{-1}}(E)$; it would disappear if $A$ was the Laplacian on $\Ga_{H^0}(E)$. 
Either way yields the same spaces $\Ga_{H^r(g)}(E)$ by \cite[Proposition~2.1]{kalton2006perturbation}, but the second construction requires higher Sobolev regularity $\al\geq 2$ instead of $\al\geq 1$.

\begin{lemma}[Fractional Laplacian]
\label{lem:fractional}
Let $\al\in (m/2,\infty)$ with $\alpha\geq 1$, let $g \in \Met_{H^\al}(M)$ and let $E$ be a natural first order vector bundle over $M$.
Then the following statements hold: 
\begin{enumerate}
\item 
\label{lem:fractional:a}
For all $r,s\in\mathbb R$, the operator $(1+\De^g)^{(s-r)/2}:  \Ga_{H^s(g)}(E)\to \Ga_{H^r(g)}(E)$
is an isometry with the same eigenfunctions $(e_i)$ and eigenvalues $(\la_i^{(s-r)/2})$ as in \cref{lem:functional}. 

\item
\label{lem:fractional:b}
For all $s \in [-\al,\al]$, the identity on $\Ga(E)$ extends to a bounded linear map $\Ga_{H^s(g)}(E)\to\Ga_{H^s}(E)$ with bounded inverse such that the following function is locally bounded:
\begin{equation*}
\Met_{H^\al}(M) \ni g \mapsto \|\on{Id}\|_{L(\Ga_{H^s(g)}(E),\Ga_{H^s}(E))} + \|\on{Id}\|_{L(\Ga_{H^s}(E),\Ga_{H^s(g)}(E))} \in \mathbb R.
\end{equation*}

\item
\label{lem:fractional:c}
If $E$ is trivial, then \ref{lem:fractional:b} holds for all $s \in [-\al,\al+1]$.
\end{enumerate}
\end{lemma}

\cref{lem:fractional} provides sufficient conditions for the equality of the fractional domain spaces and the usual Sobolev spaces; see points \ref{lem:fractional:b} and \ref{lem:fractional:c}. 
We repeat our warning that this equality may cease to hold when theses conditions are violated.

\begin{proof}
\begin{enumerate}[wide]
\item holds by the definition of the fractional domain spaces; see \cref{sec:domain}.

\item The statement follows from the following three claims: 
\begin{enumerate}[label={\bfseries Claim \arabic*:}, ref={\bfseries Claim \arabic*}, wide,  labelindent=0pt]
\item
\label{lem:fractional:claim1}
The statement holds for $s=-1$. 
This follows from \cref{lem:functional}.\ref{lem:functional:a}, noting that the operator $I^{g,\hat g}$ constructed in its proof depends continuously on $g \in \Met_{H^\al}(M)$.

\item 
\label{lem:fractional:claim2}
If the statement holds for $r \in[-\al,\al]$ and if $s=r+2k\in[-\al,\al]$ for some $k \in \mathbb Z$, then the statement holds for $s$.
To prove the claim, note that the following norms are finite and depend continuously on $g \in \Met_{H^\al}(M)$ by \ref{lem:fractional:a} and \cref{lem:functional}:
\begin{align*}
\hspace{2em}&\hspace{-2em}
\|\Id\|_{L(\Ga_{H^s(g)}(E),\Ga_{H^s}(E))}
=
\|(1+\Delta^g)^{-k}\|_{L(\Ga_{H^r(g)}(E),\Ga_{H^s}(E))}
\\&\leq
\|\Id\|_{L(\Ga_{H^r(g)}(E),\Ga_{H^r}(E))} \|(1+\Delta^g)^{-k}\|_{L(\Ga_{H^r}(E),\Ga_{H^s}(E))},
\\
\hspace{2em}&\hspace{-2em}
\|\Id\|_{L(\Ga_{H^s}(E),\Ga_{H^s(g)}(E))}
=
\|(1+\Delta^g)^k\|_{L(\Ga_{H^s}(E),\Ga_{H^r(g)}(E))}
\\&\leq
\|(1+\Delta^g)^k\|_{L(\Ga_{H^s}(E),\Ga_{H^r}(E))} \|\Id\|_{L(\Ga_{H^r}(E),\Ga_{H^r(g)}(E))},
\end{align*}

\item 
\label{lem:fractional:claim3}
If the statement holds for $s_1,s_2 \in \mathbb R$, then it holds for all $s$ in the convex hull of $\{s_1,s_2\}$. This is true because the scales of spaces $\Ga_{H^s}(E)$ and $\Ga_{H^s(g)}(E)$, $s \in \mathbb R$, are complex interpolation spaces by \cite[Theorem~7.4.4]{triebel1992theory2} and \cite[Proposition~2.2]{kalton2006perturbation}.
\end{enumerate}

\item follows by replacing the interval $[-\al,\al]$ in \ref{lem:fractional:claim2} by $[-\al,\al+1]$. 
\qedhere
\end{enumerate}
\end{proof}

Having identified the fractional domain spaces in \cref{lem:fractional}, we are ready to apply the general perturbative result in \cref{thm:smooth} to the present setting of Laplace operators associated to non-smooth metrics.

\begin{theorem}[Perturbations of functions of the Laplacian]
\label{thm:perturbations}
Let $\al\in (m/2,\infty)$ with $\al>1$, 
let $E$ be a natural first order vector bundle over $M$, 
let $r,s\in\mathbb R$ with $s,s+r\in[-\al,\al]$,
let $\ph \in (0,\pi)$,
and let $f$ be a holomorphic function on $S_\ph$ with $\sup_{\la \in S_\ph}|\la^{r/2} f(\la)|<\infty$.
Then the following map is real analytic:
\begin{equation*}
g\mapsto f(1+\Delta^g), \qquad \Met_{H^{\al}}(M) \to L(\Ga_{H^s}(E),\Ga_{H^{r+s}}(E)).
\end{equation*}
If $E$ is trivial, then this holds with $[-\al,\al]$ replaced by $[-\al,\al+1]$.
\end{theorem}

\begin{proof}
By \cref{thm:bochner} the Laplacian is real analytically as a map
\begin{equation*}
\Met_{H^\al}(M) \ni g \mapsto 1+\De^g \in L(\Ga_{H^\al}(E),\Ga_{H^{\al-2}}(E))\cap L(\Ga_{H^{-\al+2}}(E),\Ga_{H^{-\al}}(E)).
\end{equation*}
Moreover, the functional calculus is real analytic by \cref{thm:smooth}. 
Note that the conditions of the \lcnamecref{thm:smooth} are satisfied because the assumption $\al>1$ ensures that $\be:=-\al$ and $\ga:=\al-2$ satisfy $\be<\ga$.
If $E$ is trivial, then \cref{thm:bochner} holds with $[-\al,\al]$ replaced by $[-\al,\al+1]$. 
\end{proof}

Recall from \cref{lem:derivative} that the directional derivative of the Laplacian with respect to the metric extends to Sobolev spaces of low regularity. 
This also applies to fractional Laplacians, as shown in the following \lcnamecref{lem:adjoint}. 
These results are used in the proof of \cref{thm:satisfies} below.

\begin{lemma}[Derivative of the fractional Laplacian] 
\label{lem:adjoint}
Let $\al\in (m/2,\infty)$ with $\al>1$, 
let $E$ be a natural first order vector bundle over $M$, 
let $\ph \in (0,\pi)$,
and let $f$ be a holomorphic function on $S_\ph$ which satisfies for some $p\in (1,\al]$ that $\sup_{\la \in S_\ph}|\la^p f(\la)|<\infty$.
Then the derivative of $P_g=f(1+\Delta^g)$ with respect to the metric $g$ extends to a real analytic map
\begin{equation*}
\Met_{H^\al}(M) \times \Ga_{H^{2p-\al}}(S^2T^*M))\ni (g,q) \mapsto D_{g,q}P_g \in L(\Ga_{H^\al}(E),\Ga_{H^{-\al}}(E)).
\end{equation*}
\end{lemma}

\begin{proof}
Let $X,Y,Z$ be the spaces of operators given by
\begin{align*}
X&=L(\Ga_{H^\al}(E),\Ga_{H^{\al-2}}(E))\cap L(\Ga_{H^{2-\al}}(E),\Ga_{H^{-\al}}(E)),
\\
Y&=L(\Ga_{H^\al}(E),\Ga_{H^{-\al+2(p-1)}}(E))\cap L(\Ga_{H^{\al-2p+2}}(E),\Ga_{H^{-\al}}(E)),
\\
Z&=L(\Ga_{H^\al}(E),\Ga_{H^{\al-2}}(E))\cap L(\Ga_{H^{\al-2p+2}}(E),\Ga_{H^{\al-2p}}(E)).
\end{align*}
Note that the conditions $\al>1$ and $p>1$ ensure that $X$, $Y$, and $Z$ are intersections of operator spaces on distinct Sobolev scales, as required in \cref{thm:smooth}.
Moreover, let $U$ be an open neighborhood of $1+\De^g \in X$ with $g \in \Met_{H^\al}(M)$ such that the holomorphic functional calculus is well-defined and holomorphic in the sense of \cref{thm:smooth}.
Then the desired map is the composition of the following two maps: 
\begin{gather*}
\Met_{H^\al}(M)\times \Ga_{H^{2p-\al}}(S^2T^*M) \in (g,q)\mapsto (1+\De^g,D_{g,q}\De^g) \in (X,Y),
\\
(U,Y) \ni (A,B) \mapsto D_{A,B}f(A) \in L(\Ga_{H^\al}(E),\Ga_{H^{-\al}}(E)).
\end{gather*} 
The first map is real analytic by \cref{thm:bochner} and \cref{lem:derivative}.
The second map needs some interpretation. 
Note that the identity
\begin{multline*}
\frac{1}{2\pi i}\int_{\partial(S_\om\setminus\bigcirc)} f(\la) R_\la(A)BR_\la(A) d\la
\\=
A^{\al-p}\frac{1}{2\pi i}\int_{\partial(S_\om\setminus\bigcirc)} f(\la) R_\la(A)A^{p-\al}BR_\la(A) d\la.
\end{multline*}
implies that
\begin{equation*}
\forall A \in U, \forall B \in Y\cap Z: 
\qquad 
D_{A,B}f(A)
=
A^{\al-p}D_{A,A^{p-\al}B}f(A).
\end{equation*}
The right-hand side is the composition of the following maps, which are real analytic by \cref{thm:smooth}:
\begin{gather*}
(U,Y) \ni (A,B) \mapsto (A,A^{p-\al}B) \in (U,Z),
\\
(U,Z) \ni (A,B) \mapsto (A,D_{A,B}f(A)) \in U\times L(\Ga_{H^\al}(E),\Ga_{H^{\al-2p}}(E))
\\
U\times L(\Ga_{H^\al}(E),\Ga_{H^{\al-2p}}(E)) \ni (A,B) \mapsto A^{\al-p}B \in L(\Ga_{H^\al}(E),\Ga_{H^{-\al}}(E))
\end{gather*}
This shows the statement for $p \in (1,\al]$.
Finally, the statement for $f(z)=z^p$ with $p=1$ follows directly from \cref{lem:derivative}.
\end{proof}

\section{Metrics on spaces of metrics}
\label{sec:metrics} %

This section is devoted to Riemannian geometry on spaces of Riemannian metrics. 
The theory developed in the previous sections will be used to establish wellposedness of the geodesic equation for a wide class of metrics, which are defined via the functional calculus of Laplace operators. 
Our main results, \cref{thm:satisfies,thm:wellposed} below,  close a gap in an earlier proof in \cite{Bauer2013c} for integer order metrics and generalize this result to a much wider class of metrics, including Sobolev metrics of fractional order. 

\subsection{Weak Riemannian metrics on \texorpdfstring{$\Met(M)$}{Met(M)}}
\label{sec:weak}

We consider $\Diff(M)$-invariant Riemannian metrics on $\Met(M)$ of the form 
\begin{align*}
G^P_g(h,k)&%
=\int_M \on{Tr}(g\i (P_g h) g\i k) \vol(g),
\end{align*}
\newline
where for each $g\in\Met(M)$, the operator
\begin{equation*}
P_g \colon \Ga(S^2T^*M) \to \Ga(S^2T^*M)
\end{equation*}
is positive and symmetric with respect to the $H^0(g)$ inner product, i.e., 
\begin{align*}
\int_M \on{Tr}(g\i (P_g h) g\i h) \vol(g)>0, \text{ iff } h\neq 0\in T_g\Met(M)\\
\int_M \on{Tr}(g\i (P_g h) g\i k) \vol(g)=\int_M \on{Tr}(g\i  h g\i P_g k) \vol(g);.
\end{align*}
Note, that these two conditions are necessary in order to ensure that $G^P$ is a Riemannian metric.
In the following we will refer to $P$ as an operator field.
Further conditions on this operator field are formulated in \cref{sec:conditions} below.

Including the operator field $P$ in the definition of the metric 
allows one to consider higher order metrics. In particular 
the setting encompasses the following examples:
\begin{align*}
G^P_g(h,k)&%
 =\int \on{Tr}(g\i h g\i k)\vol(g),\quad H^0\text{-metric}
\\
\text{or  }&=\int_M \on{Tr}\Big(g\i h g\i (1+\De^g)^p k\Big) \vol(g)  \qquad\text{Sobolev $H^p$ metric, }p\in \mathbb R_{>0}
\\
\text{or  }&=\int_M\on{Tr}\Big(g\i h g\i f(1+\De^g) k\Big) \vol(g)
\end{align*}
where 
$f$ is a suitable spectral function as considered in \cref{sec:sectorial,sec:perturbative}. 
Further metrics considered in the literature include curvature and volume weighted metrics, which can also be formulated in the present framework \cite{Clarke2013c}.

\begin{condition}
\label{sec:conditions}
We will frequently use the following conditions on the operator field $P$ and non-negative real numbers $\al,p$ with $\al>m/2$:
\begin{enumerate}
\item
\label{sec:conditions:a}
The operator field $P$ is smooth as a map
\begin{equation*}
\Met_{H^\al}(M)\ni g \mapsto P_g \in GL(\Ga_{H^\al}(S^2T^*M),\Ga_{H^{\al-2p}}(S^2T^*M)),	
\end{equation*}
where $GL$ denotes bounded linear operators with bounded inverse. 

\item 
\label{sec:conditions:b}
The operator field $P$ is $\Diff(M)$-equivariant in the sense that one has for all $\ph\in\Diff(M)$, $g \in \Met_{H^\al}(M)$, and $h\in\Ga_{H^\al}(S^2T^*M)$ that
\begin{equation*}
\ph^*(P_gh) = P_{\ph^*g}(\ph^*h).
\end{equation*}

\item 
\label{sec:conditions:c}
For each $g\in \Met_{H^\al}(M)$, the operator $P_g$ is nonnegative and symmetric with respect to the $H^0(g)$ inner product on $\Ga_{H^\al}(S^2T^*M)$, i.e., for all $h,k \in \Ga_{H^\al}(S^2T^*M)$:
\begin{align*}
&\int_M \on{Tr}(g\i P_g h g\i h)\vol(g) \geq 0,\\
&\int_M \on{Tr}(g\i P_g h g\i k)\vol(g) = \int_M \on{Tr}(g\i h g\i P_g k)\vol(g).
\end{align*}

\item
\label{sec:conditions:d}
The $H^0(g)$ adjoint of the derivative of $P$ with respect to the metric is well-defined as a smooth map
\begin{multline*}
\Met_{H^\al}(M)\times \Ga_{H^\al}(S^2T^*M)
\ni (g,h) \mapsto (D_{(g,\cdot)}P_gh)^*
\\
\in L(\Ga_{H^\al}(S^2T^*M),\Ga_{H^{\al-2p}}(S^2T^*M))
\end{multline*} 
such that the following relation is satisfied for all $g \in \Met_{H^\al}(M)$ and $h,k\in\Ga_{H^\al}(S^2T^*M)$:
\begin{equation*}
\int_M \on{Tr}\Big( g\i \big((D_{(g,q)}P_g)h\big) g\i k \Big)\vol(g)
=
\int_M \on{Tr}\big(g\i q g\i (D_{(g,\cdot)}P_gh)^*(k)\big) \vol(g).
\end{equation*}
\end{enumerate}
\end{condition}

\begin{remark}
In \cite[Section~3.2]{Bauer2013c} we had more complicated conditions, and we implicitly claimed that they imply \cref{sec:conditions} above.
There was, however, a significant gap in the argumentation of the main result.
Namely, we did not show the smoothness of the extended mappings on Sobolev completions.
The results of this article allow us to close this gap and to extend the analysis to the larger class of fractional order metrics.
\end{remark}

The following \lcnamecref{thm:satisfies} provides a wide class of operators which satisfy \cref{sec:conditions}:

\begin{theorem}[Conditions on \texorpdfstring{$P$}{P}]
\label{thm:satisfies}
Let $\ph\in(0,\pi)$, 
let $p \in (1,\infty)$, 
and let $f$ be a holomorphic function on the sector $S_\ph$ which satisfies for some constant $C>0$ that
\begin{equation*}
\forall z \in S_\ph: \qquad C\i |z^p| \le |f(z)|\le C |z^p|.
\end{equation*}
Then the field of operators
\begin{equation*}
\Met (M) \ni g \mapsto P_g:=f(1+\De^g)\in L(\Ga(S^2T^*M), \Ga(S^2T^*M))
\end{equation*}
satisfies \cref{sec:conditions} for any $\al \in (m/2,\infty)\cap[p,\infty)$. 
For $f(z)=z^p$ (i.e., the fractional Laplacian) the \lcnamecref{thm:satisfies} continues to hold for $p=1$.
\end{theorem}

\begin{proof}
We shall check \cref{sec:conditions}.\ref{sec:conditions:a}--\ref{sec:conditions:d} one by one. 
\begin{enumerate}[wide]
\item follows from \cref{thm:perturbations} applied to the functions $f$ and $f\i$.

\item 
$\Diff(M)$-invariance of $(1+\De^g)$ is well-known for smooth $g$ and follows in the general case by approximation, noting that the pull-back along a smooth diffeomorphism is a bounded linear map between Sobolev spaces of the same order of regularity \cite[Theorem~B.2]{inci2013regularity}.
By the resolvent integral representation of the functional calculus this implies $\Diff(M)$-invariance of $f(1+\De^g)$.

\item is well-known for smooth $g,h,k$ and follows in the general case by approximation using the continuity of $g\mapsto\langle\cdot,\cdot\rangle_{H^0(g)}$ established in \cref{lem:volume} and the continuity of $g\mapsto P_g$.

\item 
By \cref{lem:adjoint} the derivative of $P$ with respect to the metric extends to a smooth map
\begin{equation*}
\Met_{H^\al}(M) \times \Ga_{H^{2p-\al}}(S^2T^*M)\ni (g,q) \mapsto D_{g,q}P_g \in L(\Ga_{H^\al}(E),\Ga_{H^{-\al}}(E)).
\end{equation*}
Equivalently, by changing the order of the arguments $g$, $q$, and $h$, the following function is smooth:
\begin{equation*}
\Met_{H^\al}(M) \times \Ga_{H^\al}(E)\ni (g,h) \mapsto D_{(g,\cdot)}P_gh \in L(\Ga_{H^{2p-\al}}(S^2T^*M),\Ga_{H^{-\al}}(E)).
\end{equation*}
Dualization using the $H^0(g)$ duality shows that the adjoint is smooth
\begin{equation*}
\Met_{H^\al}(M) \times \Ga_{H^\al}(E)\ni (g,h) \mapsto (D_{(g,\cdot)}P_gh)^* \in L(\Ga_{H^{\al}}(E),\Ga_{H^{\al-2p}}(S^2T^*M)).
\end{equation*}
\qedhere
\end{enumerate}
\end{proof}

Next, we derive the geodesic equation of the metric $G^P$:
\begin{theorem}[The geodesic equation]\label{sec:geodesic}
In the notation of \cref{sec:weak}, the geodesic equation of the metric $G^P$ reads as 
\begin{multline*}
g_{tt}
=P_g\i\Big[\frac12(D_{(g,\cdot)}P_gg_t)^*(g_t)+\frac14 g \on{Tr}(g\i(P_gg_t)g\i g_t)
\\
+\frac12 g_t g\i (P_gg_t)+\frac12 (P_gg_t) g\i g_t-(D_{(g,g_t)}P_g)g_t
-\frac12\on{Tr}(g\i g_t) (P_gg_t)\Big],
\end{multline*}
where the subscripts $t$ denote differentiation in time, and $(D_{(g,\cdot)}P_gg_t)^*$ is the adjoint defined in \cref{sec:conditions}.\ref{sec:conditions:d}.
\end{theorem}

The above formula has been first obtained in~\cite[Section~3.1]{Bauer2013c} using a Hamiltonian approach. To keep the presentation self-contained we present a direct derivation.

\begin{proof}
The kinetic energy of a path $g\colon [0,1]\to \Met(M)$ is defined as 
\begin{equation*}
E(g)=\frac{1}{2}\int_0^1 G^P_g(g_t,g_t) dt.
\end{equation*}
The variation in the direction of a path $h\colon[0,1]\to\Ga(S^2T^*M)$ with vanishing end points is given by
\begin{align*}
D_{g,h} E
&=\int_0^1 \int_M\Bigg(\Tr\big((D_{g,h}g\i)P_gg_tg\i g_t\big)+ \Tr\big(g\i P_gg_tg\i h_t\big)\\
&\quad+\frac{1}{2} \Tr\big(g\i(D_{(g,h)}P)g_tg\i g_t\big)+\frac{1}{2}\Tr\big(g\i Pg_tg\i g_t\big)\frac{D_{g,h}\vol(g)}{\vol(g)}\Bigg)\vol(g)dt.
\end{align*}
Using the variation formulas \cite{Bauer2013c}
\begin{align*}
D_{g,h}g\i=-g\i h g\i,\qquad \frac{D_{g,h}\vol(g)}{\vol(g)}=\Tr(g\i h)
\end{align*}
we obtain
\begin{align*}
&D_{g,h} E=\int_0^1 \int_M\Bigg(\Tr\big(g\i hg\i P_gg_tg\i g_t\big)+ \Tr\big(g\i P_gg_tg\i h_t\big)\\
&\quad\quad+\frac{1}{2} \Tr\big(g\i(D_{g,h}P)g_tg\i g_t\big)+\frac{1}{2}\Tr\big(g\i Pg_tg\i g_t\big)\Tr(g\i h)\Bigg)\vol(g)dt.
\end{align*}
Integrating by parts in the variable $t$ of the second term and using the definition of the adjoint for the third term allows one to write the variation of the energy as
\begin{align*}
&D_{g,h} E=\int_0^1 \int_M\Tr\big(g\i hg\i P_g(g_{tt}-\Gamma_g(g_t,g_t))\big)\vol(g)dt,
\end{align*}
where the Christoffel symbol $\Gamma_g(g_t,g_t)$ is given by the right-hand side of the geodesic equation in the statement of the \lcnamecref{sec:geodesic}.
\end{proof}

We will show well-posedness of the geodesic equation using the Ebin--Marsden \cite{EM1970} approach of extending the geodesic spray to a smooth vector field on  $T\Met_{H^{\alpha}}(M)$ for sufficiently high $\al$ and showing that solutions exist on an interval which is independent of $\al$. 
The latter statement is a consequence of the no-loss-no-gain theorem of \cite{EM1970}, which we adapt to the present setting in the following \lcnamecref{lem:nolossnogain}. 

\begin{lemma}[No-loss-no-gain]
\label{lem:nolossnogain}
Let $\al>m/2$,
let $S$ be a smooth $\Diff(M)$-invariant vector field on $T\Met_{H^\al}(M)$,
let $T\in(0,\infty]$, 
let $U$ be an open $\Diff(M)$-invariant subset of $T\Met_{H^\al}(M)$,
and assume that the flow of $S$ exists as a smooth map
\begin{equation*}
\on{Fl}^S\colon [0,T) \times U \to T\Met_{H^\al}(M).
\end{equation*} 
Then the flow restricts to a smooth map
\begin{equation*}
\on{Fl}^S\colon [0,T) \times (U \cap T\Met_{H^{\al+1}}(M)) \to T\Met_{H^{\al+1}}(M).
\end{equation*} 
Thus, there is no loss or gain in regularity during the evolution along $S$.
\end{lemma}

\begin{proof}
The proof is divided in two steps.
\begin{enumerate}[wide]
\item 
\label{lem:nolossnogain:a}
We claim that there is a finite number $n\in\mathbb N$ and  vector fields $X_1,\dots,X_n$ such that $\Ga_{H^{\al+1}}(S^2T^*M)$ carries the initial topology with respect to the map
\begin{equation*}
\Ga_{H^{\al+1}}(S^2T^*M) \ni h \mapsto (h,\mathcal L_{X_1}h,\dots,\mathcal L_{X_n}h) \in \Ga_{H^\al}(S^2T^*M)^{n+1},
\end{equation*}  
where $\mathcal L$ denotes the Lie derivative.
Loosely speaking, this means that $h$ has regularity $H^{\al+1}$ whenever $h$ and its Lie derivatives have regularity $H^\al$.
The claim can be shown by adapting the proof of \cite[Lemma~12.2]{EM1970} to diffeomorphisms acting on Riemannian metrics by pull backs.
This task is facilitated by the fact that the vector fields in the present setting are not required to be divergence free.
The key observation is that in any chart, Lie derivatives along coordinate vector fields coincide with ordinary derivatives. 
Moreover, the charts can be constructed as in \cref{sec:sobolev} such that the coordinate vector fields extend to smooth vector fields on all of $M$.
Thus, the claim follows from the well-known fact that the space $H^{\al+1}(\mathbb R^m,\mathbb R^{m(m+1)/2})$ carries the initial topology with respect to the map
\begin{equation*}
H^{\al+1}(\mathbb R^m,\mathbb R^{m(m+1)/2}) \ni h \mapsto (h,\partial_{x^1}h,\dots,\partial_{x^m}h) \in H^\al(\mathbb R^m,\mathbb R^{m(m+1)/2})^{m+1}.
\end{equation*}

\item 
The rest of the proof is as in \cite[Theorem 12.1]{EM1970}.
Let $X \in \mathfrak X(M)$ be a smooth vector field, 
and let $\mathbb R\ni s \mapsto \ph_s:=\on{Fl}^X_s \in \Diff(M)$ be the flow of $X$ on $M$.
As the $\Diff(M)$-equivariance of $S$ implies the $\Diff(M)$-equivariance of $\on{Fl}^S$, one obtains for any $s\in\mathbb R_{\geq 0}$, $t \in [0,T]$, and $(g,h)\in T\Met_{H^{\al+1}}(M)$ that
\begin{equation*}
\ph_s^*\big(\on{Fl}^S_t(g,h)\big)=\on{Fl}^S_t\big(\ph_s^*(g,h)\big).
\end{equation*}
Differentiating this equation with respect to $s$ and evaluating at $s=0$ yields 
\begin{equation*}
\mathcal L_X\big(\on{Fl}^S_t(g,h)\big) 
= 
T\on{Fl}^S_t\big(\mathcal L_X(g,h)\big).
\end{equation*}
The right-hand side, seen as a function of $t$, is a smooth curve in $TT\Met_{H^\al}(M)$ thanks to the $H^{\al+1}$ regularity of $(g,h)$ and the smoothness of $\on{Fl}^S$ in the $H^\al$ topology.
Thus, the left-hand side enjoys the same regularity, and it follows from \ref{lem:nolossnogain:a} that $t\mapsto\on{Fl}^S_t(g,h)$ is a smooth curve in $T\Met_{H^{\al+1}}(M)$.
\qedhere
\end{enumerate} 
\end{proof}

We are now able to prove the main result of this section, namely local well-posedness of the geodesic equation under \cref{sec:conditions}, which is satisfied for fractional order Sobolev metrics by \cref{thm:satisfies}.

\begin{theorem}[Well-posedness of the geodesic equation]
\label{thm:wellposed}
Assume that the operator $P$ satisfies \cref{sec:conditions} for some $p\in \mathbb R_{\geq 0}$ and all $\al \in [\al_0,\infty)$ with $\al_0\in (m/2,\infty)$.
Then the following statements hold for each $\al \in [\al_0,\infty)$.
\begin{enumerate}
\item 
\label{thm:wellposed:a}
The initial value problem for the geodesic equation has unique local solutions in $\Met_{H^\al}(M)$. 
The solutions depend smoothly on $t$ and on the initial conditions $g(0)\in \Met^{\al}(M)$ and $g_t(0)\in \Ga_{H^\al}(S^2T^*M)$.

\item
\label{thm:wellposed:b}
The Riemannian exponential map $\exp^{P}$ exists and is smooth on a neighborhood of the zero section in $T\Met_{H^\al}(M)$, 
and $(\pi,\exp^{P})$ is a diffeomorphism from a (smaller) neighborhood of the zero section to a neighborhood of the diagonal in $\Met^{\al}(M)\x \Met^{\al}(M)$. 

\item
\label{thm:wellposed:c}
The neighborhoods in \ref{thm:wellposed:a}--\ref{thm:wellposed:b} are uniform in $\al$ and can be chosen open in the $H^{\al_0}$ topology.
Thus, \ref{thm:wellposed:a}--\ref{thm:wellposed:b} continue to hold for $\al=\infty$, i.e., on the Fr\'echet manifold $\Met(M)$ of smooth metrics.
\end{enumerate}
\end{theorem}

\begin{proof}
\begin{enumerate}[wide]
\item
This can be shown as in \cite[Theorem~3.2]{Bauer2013c}. 
For the convenience of the reader we repeat the proof in the notation of the present paper.
The geodesic equation can be written as 
\begin{align*}
g_t &= S_1(g,h) := h\\
h_t&= S_2(g,h) := P_g^{-1}\Big(\frac12(D_{(g,\cdot)}P_gg_t)^*(g_t)+\frac14 g \on{Tr}(g\i(P_gg_t)g\i g_t)\\
&\qquad+ \frac12 g_t g\i(P_gg_t)+\frac12 (P_gg_t)g\i g_t-\frac12\on{Tr}(g\i g_t) (P_gg_t)\Big).
\end{align*}
This is the flow equation of the geodesic spray $S=(S_1,S_2)$, which is a vector field on the tangent space $T\Met_{H^{\alpha}}(M) = \Met_{H^{\alpha}}(M)\times \Ga_{H^\al}(S^2T^*M)$. 
For any $\al>\al_0$, a term by term investigation of the right-hand side using \cref{sec:conditions}.\ref{sec:conditions:c}--\ref{sec:conditions:d} shows that $S$ is a smooth vector field on $T\Met_{H^{\alpha}}(M)$.
Thus, the theorem of Picard-Lindel\"of shows that the flow of $S$ exists as a smooth map
\begin{equation*}
\on{Fl}^S\colon[0,T)\times U \to T\Met_{H^{\alpha}}(M)
\end{equation*}
for some $T>0$ and some open subset $U$ of $T\Met_{H^{\alpha}}(M)$, which may be chosen $\Diff(M)$-invariant thanks to the $\Diff(M)$-equivariance of $S$.

\item follows from \ref{thm:wellposed:a} as in \cite[Theorem~3.2]{Bauer2013c}, and \ref{thm:wellposed:c} follows from \cref{lem:nolossnogain}.
\qedhere
\end{enumerate}
\end{proof}

\bibliography{immersions}
\bibliographystyle{abbrv}

\end{document}